\documentclass[12pt]{article}
\usepackage{mathrsfs}
\usepackage{graphicx}
\usepackage{epstopdf}
\usepackage{amsmath}
\usepackage{amsfonts}
\usepackage{amssymb, amsmath, cite}
\usepackage{cases}
\usepackage{appendix}

\newtheorem{theorem}{Theorem}
\newtheorem{lemma}{Lemma}
\newtheorem{proposition}{Proposition}
\newtheorem{remark}{Remark}

\newbox\qedbox
\newenvironment{proof}{\smallskip\noindent{\bf Proof.}\hskip \labelsep}%
                        {\hfill\penalty10000\copy\qedbox\par\medskip}

\newcommand{\bfR}{{\Bbb R}}
\newcommand{\bfC}{{\Bbb C}}

\newcommand{\ii}{\text{i}}
\newcommand{\e}{\text{e}}
\newcommand{\dd}{\text{d}}
\newcommand{\Om}{\Omega}

\newcommand{\nn}{\nonumber}
\newcommand\be{\begin{equation}}
\newcommand\ee{\end{equation}}
\newcommand{\bea}{\begin{eqnarray}}
\newcommand{\eea}{\end{eqnarray}}
\newcommand\berr{\begin{eqnarray*}}
\newcommand\eerr{\end{eqnarray*}}
{\setlength\arraycolsep{2pt}

\begin{document}

\title{Long-time asymptotics for the Sasa--Satsuma equation via nonlinear steepest descent method}
\author{ Boling Guo$^{a}$,\, Nan Liu$^{b,}$\footnote{Corresponding author.},\, Yufeng Wang$^{c}$\\
$^a${\small{\em Institute of Applied Physics and Computational Mathematics,  Beijing 100088, P.R. China}} \\
$^b${\small{\em The Graduate School of China Academy of Engineering Physics, Beijing 100088, P.R. China}}\\$^c${\small{\em College of Science, Minzu University of China, Beijing 100081, P.R. China}}\setcounter{footnote}{-1}\footnote{E-mail addresses: gbl@iapcm.ac.cn (B. Guo), ln10475@163.com (N. Liu), yufeng\_0617@126.com (Y. Wang).} \\
}

\date{}
\maketitle

\begin{quote}
{{{\bfseries Abstract.} We formulate a $3\times3$ Riemann--Hilbert problem to solve the Cauchy problem for the Sasa--Satsuma equation on the line, which allows us to give a representation for the solution of Sasa--Satsuma equation. We then apply the method of nonlinear steepest descent to compute the long-time asymptotics of the Sasa--Satsuma equation.

}

 }
\end{quote}

\section{Introduction}
\setcounter{equation}{0}

In the context of inverse scattering, the first work to provide explicit formulas (i.e., depending only on initial conditions) for large-time asymptotics of solutions is due to Zakharov and Manakov \cite{ZM} in the context of the nonlinear Schr\"odinger (NLS) equation. In this setting, the inverse scattering map and the reconstruction of the solution (potential) is formulated through an oscillatory Riemann--Hilbert (RH) problem. Then the now well-known nonlinear steepest descent method introduced by Deift and Zhou in \cite{PD} provides a detailed rigorous proof to calculate the large-time asymptotic behaviors of the integrable nonlinear evolution equations. This approach has been successfully applied in determining asymptotic formulas for the initial value problems of a number of integrable systems associated with $2\times2$ matrix sprectral problems including the mKdV equation \cite{PD}, the defocusing NLS equation \cite{PD1}, the KdV equation \cite{GT}, the Hirota equation \cite{HL}, the derivative NLS equation \cite{LP}, the Fokas--Lenells equation \cite{XJ} and the Kundu--Eckhaus equation \cite{DS}. Moreover, by combining the ideas of \cite{PD} with the so-called ``$g$-function mechanism" \cite{PD2}, it is also possible to study asymptotics of solutions of the initial value problems with shock-type oscillating initial data \cite{RB}, nondecaying step-like initial data \cite{AB3,XJ3,KM} and the initial-boundary value problems with $t$-periodic boundary condition \cite{AB,ST} for various integrable equations. Recently, Lenells et al. also has been derived some interesting asymptotic formulas for the solution of derivative NLS equation on the half-line under the initial and boundary values lie in the Schwartz class \cite{AL} by using the steepest descent method. However, there is just a little literature \cite{AB1,AB2,XG} about the study of long-time asymptotics for the integrable nonlinear evolution equations associated with $3\times3$ matrix spectral problems. Therefore, it is necessary and important to consider the large-time asymptotic behavior of the integrable equations with $3\times3$ Lax pairs.

Our present paper aim to consider the long-time asymptotics for the initial value problem of the Sasa--Satsuma (SS) equation \cite{SS},
\begin{equation}\label{1.1}
u_t=u_{xxx}+6|u|^2u_x+3u(|u|^2)_x,
\end{equation}
with the initial datum
\be
u(x,0)=u_0(x),
\ee
where $u_0(x)$ belongs to the Schwartz space $S(\bfR)$. The SS equation derived in \cite{SS} is of considerable interest for applications and is widely used in nonlinear optics (see \cite{KP} and references therein) because the integrable cases of the so-called higher order NLS equation \cite{YK,YK1} describing the propagation of short pulses in optical fibers are related through a gauge transformation either to the SS equation \cite{SS} or to the so-called Hirota equation \cite{HR}. Due to the important role played in both physics and mathematics, the SS equation has attracted much attention and various works were presented. For example, by employing $3\times3$ Lax pair, the inverse scattering transform formalism for the initial value problem of the SS equation has been obtained in \cite{SS}. Soliton solution, B\"acklund transformation, and conservation laws for the SS equation were found in the optical fiber communications \cite{LY}. Twisted rogue-wave pairs in the SS equation were got by the author in \cite{CSH}. Squared eigenfunctions are derived for the SS equation in \cite{YJK}. The initial-boundary value problem of the SS equation on the half-line was analysed in \cite{JX1} by using the Fokas method. There also exists some interesting results for the so-called coupled SS equations, we refer the readers to see \cite{ZHQ,AS,WJ,KN,LX}.

Our purpose here is to derive the long-time asymptotics of the solution $u(x,t)$ of SS equation \eqref{1.1} with a $3\times3$ Lax pair on the line by performing a nonlinear steepest descent analysis of the associated RH problem. Developing and extending the unified transform approach announced in \cite{FK,JL}, our primary task is to formulate the main RH problem corresponding to the equation \eqref{1.1}. As a result, we can give a representation of the solution to the Cauchy problem \eqref{1.1} in terms of the solution of the corresponding $3\times3$ RH problem with the jump matrix $J(x,t,k)$ given in terms of the scattering matrix $s(k)$. By using the perfect symmetry of the jump matrix, we introduce a $1\times2$ vector-valued spectral function $\rho(k)$ and rewrite our main RH problem as a $2\times2$ block ones. This procedure is more convenient for the following long-time asymptotic analysis compared with the analysis in \cite{AB1,AB2}. As we all known, the first important step of the steepest descent method is to split the jump matrix $J(x,t,k)$ into an appropriate upper/lower triangular form. This immediately leads to construct a $\delta(k)$ function to remove the middle matrix term, however, the function $\delta$ satisfies a $2\times2$ matrix RH problem in our present problem. The unsolvability of the $2\times2$ matrix function $\delta(k)$ is a challenge for us when we perform the scaling transformation to reduce the RH problem to a model RH problem. Fortunately, the function $\det\delta(k)$ can be explicitly solved by the Plemelj formula because $\det\delta(k)$ satisfies a scalar RH  problem. Therefore, we follow the idea introduced in \cite{XG} to use the available function $\det\delta(k)$ to approximate the function $\delta(k)$ by error control.
On the other hand, the spectral curve for SS equation \eqref{1.1} possesses two stationary points, which different from the case of coupled NLS system considered in \cite{XG} where the phase function has a single critical point. The symmetry of the spectral function $\rho(k)$ plays an important role in study the solution of the model RH problem near the critical point $-k_0$. Therefore, the study of the long-time asymptotics for the initial value problem of \eqref{1.1} on the line is more involved. These are some innovation points of the present paper.

The organization of this paper is as follows. In Section 2, we formulate the main RH problem and show how the solution of the SS equation \eqref{1.1} can be expressed in terms of the solution of the $3\times3$ matrix RH problem. In Section 3, we transform the original RH problem to a suitable form and derive the long-time asymptotic behavior of the solution of the SS equation \eqref{1.1}.

\section{Riemann--Hilbert problem}
\setcounter{equation}{0}
\setcounter{lemma}{0}
\setcounter{theorem}{0}
In this section, we aim to formulate a RH problem to solve the Cauchy problem of the SS equation \eqref{1.1}.
\subsection{The Lax pair and basic eigenfunctions}
The Lax pair of equation \eqref{1.1} is
\begin{equation}\label{2.1}
\left\{
\begin{aligned}
&\psi_x(x,t;k)=-\ii k\sigma\psi(x,t;k)+U(x,t)\psi(x,t;k),\\
&\psi_t(x,t;k)=4\ii k^3\sigma\psi(x,t;k)+V(x,t;k)\psi(x,t;k),\\
\end{aligned}
\right.
\end{equation}
where $\psi(x,t;k)$ is a $3\times3$ matrix-valued function, $k\in\bfC$ is the spectral parameter, and
\begin{equation}\label{2.2}
\sigma={\left( \begin{array}{ccc}
1 ~& 0 ~& 0\\
0 ~& 1 ~& 0\\
0 ~& 0 ~& -1\\
\end{array}
\right )},\quad
U(x,t)={\left( \begin{array}{ccc}
0 & 0 & u(x,t) \\
 0 & 0 & \bar{u}(x,t)\\
-\bar{u}(x,t) & -u(x,t)  & 0\\
\end{array}
\right )},
\end{equation}
\begin{equation}\label{2.3}
V(x,t;k)=4(|u|^2-k^2)U(x,t)-2\ii kV_1(x,t)+V_2(x,t),
\end{equation}
with
\begin{equation}
V_1={\left( \begin{array}{ccc}
|u|^2 ~& u^2 ~& u_x\\
\bar{u}^2 ~& |u|^2 ~& \bar{u}_x\\
\bar{u}_x ~& u_x ~& -2|u|^2\\
\end{array}
\right )},\quad
V_2={\left( \begin{array}{ccc}
u_x\bar{u}-u\bar{u}_x ~& 0 ~& u_{xx} \\
 0 ~& u\bar{u}_x-u_x\bar{u} ~& \bar{u}_{xx}\\
-\bar{u}_{xx} ~& -u_{xx}  ~& 0\\
\end{array}
\right )}.\nn
\end{equation}
Introducing a new eigenfunction $\mu(x,t;k)$ by
\begin{equation}
\psi(x,t;k)=\mu(x,t;k) \e^{-\ii(kx-4k^3t)\sigma}\nn,
\end{equation}
we obtain the equivalent Lax pair
\begin{equation}\label{2.4}
\left\{
\begin{aligned}
&\mu_x(x,t;k)+\ii k[\sigma,\mu(x,t;k)]=U(x,t)\mu(x,t;k),\\
&\mu_t(x,t;k)-4\ii k^3[\sigma,\mu(x,t;k)]=V(x,t;k)\mu(x,t;k).\\
\end{aligned}
\right.
\end{equation}

We define two eigenfunctions $\{\mu_j\}_1^2$ of $x$-part of equation \eqref{2.4} by the following Volterra integral equations
\bea
&&\mu_1(x,t;k)=I+\int_{-\infty}^x\e^{-\ii k(x-x')\hat{\sigma}}[U(x',t)\mu_1(x',t;k)]\dd x',\label{2.5}\\
&&\mu_2(x,t;k)=I-\int^{+\infty}_x\e^{-\ii k(x-x')\hat{\sigma}}[U(x',t)\mu_2(x',t;k)]\dd x',\label{2.6}
\eea
where $\hat{\sigma}$ denote the operators which act on a $3\times3$ matrix $X$ by $\hat{\sigma}X=[\sigma,X]$, then $\e^{\hat{\sigma}}X=\e^{\sigma} X\e^{-\sigma}$.
Thus, it can be shown that the functions $\{\mu_j\}^2_1$ are bounded and analytical for $k\in\bfC$ while $k$ belongs to
\begin{eqnarray}\label{2.7}
\begin{aligned}
&&\mu_1:~(\bfC_+,\bfC_+,\bfC_-),\\
&&\mu_2:~(\bfC_-,\bfC_-,\bfC_+),\\
\end{aligned}
\end{eqnarray}
where $\bfC_+$ and $\bfC_-$ denote the upper and lower half complex $k$-plane, respectively. And if we let $-\ii k\sigma=\mbox{diag}(z_1,z_2,z_3)$, then we have
\begin{eqnarray*}
&&\bfC_+=\{k\in {\bfC}|\mbox{Re} z_1=\mbox{Re} z_2>\mbox{Re}z_3\},\\
&&\bfC_-=\{k\in{\bfC}|\mbox{Re}z_1=\mbox{Re}z_2<\mbox{Re}z_3\}.
\end{eqnarray*}

The solutions of the system of differential equation \eqref{2.4} must be related by a
matrix independent of $x$ and $t$, therefore,
\bea\label{2.8}
\mu_1(x,t;k)=\mu_2(x,t;k)\e^{-\ii(kx-4k^3t)\hat{\sigma}}s(k),
\eea
Evaluation at $x\rightarrow+\infty,t=0$ gives
\be\label{2.9}
s(k)=\lim_{x\rightarrow+\infty}\e^{\ii kx\hat{\sigma}}\mu_1(x,0;k),
\ee
that is,
\be\label{2.10}
s(k)=I+\int_{-\infty}^{+\infty}\e^{\ii kx\hat{\sigma}}[U(x,0)\mu_1(x,0;k)]\dd x.
\ee
In fact, the matrix-valued spectral function $s(k)$ can be determined in terms of the initial value $u_0(x)$.

The fact that $U(x,t)$ is traceless together with equations \eqref{2.5}-\eqref{2.6} implies $\mbox{det}\mu_j(x,t;k)=1$ for $j=1,2$. Thus, $\mbox{det}s(k)=1$. On the other hand, if we denote $\tilde{U}(x,t;k)=-\ii k\sigma+U(x,t)$, then we find that $$\tilde{U}^\dagger(x,t;\bar{k})=-\tilde{U}(x,t;k),\quad \overline{\tilde{U}(x,t;-\bar{k})}=\Gamma U(x,t;k)\Gamma,$$
where
$$\Gamma=\begin{pmatrix}
0~& 1 ~& 0\\
1~& 0 ~& 0\\
0~& 0~& 1
\end{pmatrix}.$$
Thus, we get
\be\label{2.11}
\mu_j^\dagger(x,t,\bar{k})=\mu_j^{-1}(x,t,k),\quad \mu_j(x,t;k)=\Gamma\overline{\mu_j(x,t;-\bar{k})}\Gamma,~ j=1,2,
\ee
where ``$\dagger$" denotes Hermitian conjugate.

Therefore,
\be\label{2.12}
s^\dag(\bar{k})=s^{-1}(k),\quad s(-k)=\Gamma\overline{s(\bar{k})}\Gamma.
\ee
\subsection{The formulation of Riemann--Hilbert problem}
In order to formulate a RH problem, we should define the following eigenfunctions.
For each $n=1,2$, a solution $M_n(x,t;k)$ of $x$-part of \eqref{2.4} is defined by the following system of integral equations:
\be\label{2.13}
(M_n)_{jl}(x,t;k)=\delta_{jl}+\int^x_{\infty^n_{jl}}\bigg(\e^{-\ii k (x-x')\hat{\sigma}}[U(x',t)M_n(x',t;k)]\bigg)_{jl}\dd x',~k\in D_n,~j,l=1,2,3,
\ee
where $D_1=\bfC_+,D_2=\bfC_-$, and the limits $\infty^n_{jl}$ of integration, $n=1,2$, $j,l=1,2,3,$ are defined by
\be\label{2.14}
\infty^n_{jl}=\left\{
\begin{aligned}
&+\infty,~\mbox{if}~\mbox{Re}z_j(k)>\mbox{Re}z_l(k),\\
&-\infty,~\mbox{if}~\mbox{Re}z_j(k)\leq\mbox{Re}z_l(k).\\
\end{aligned}
\right.
\ee
According to the definition of the $\infty^n$, we find that
\begin{gather}\label{2.15}
\infty^1={\left( \begin{array}{ccc}
-\infty ~& -\infty ~& +\infty \\
-\infty ~& -\infty ~& +\infty \\
-\infty ~& -\infty ~& -\infty \\
\end{array}
\right )},\quad
\infty^2={\left( \begin{array}{ccc}
 -\infty ~& -\infty ~& -\infty\\
 -\infty ~& -\infty ~& -\infty\\
+\infty ~& +\infty ~& -\infty\\
\end{array}
\right )}.
\end{gather}

According to the similar proof in \cite{JL}, we have the following proposition of $M_n$.
\begin{proposition}\label{pro2.1}
For each $n=1,2$, the function $M_n(x,t;k)$ is well defined by equation \eqref{2.13} for $k\in \bar{D}_n$. $M_n$ is bounded and analytical as a function of $k\in D_n$ away from a possible discrete set of singularities $\{k_j\}$ at which the Fredholm determinant vanishes. Moreover, $M_n$ admits a bounded and continuous extension to $\bar{D}_n$ and
\be\label{2.16}
M_n(x,t;k) =I+O\bigg(\frac{1}{k}\bigg),~~k\rightarrow\infty,~~k\in D_n.
\ee
\end{proposition}

For each $n=1,2$, we define spectral functions $S_n(k)$ by
\be\label{2.17}
M_n(x,t;k)=\mu_2(x,t;k)\e^{-\ii(kx-4k^3t)\hat{\sigma}}S_n(k),~~k\in D_n,~~n=1,2.
\ee

According to \eqref{2.13}, the $\{S_n(k)\}^2_1$ can be computed from the spectral function $s(k)$.

\begin{proposition}\label{pro2.2}
The $S_n(k)$ defined in \eqref{2.17} can be expressed in terms of the entries of $s(k)$ as follows:
\bea\label{2.18}
&&S_1(k)=\begin{pmatrix}
s_{11} ~& s_{12} ~& 0\\[4pt]
s_{21} ~& s_{22} ~& 0 \\[4pt]
s_{31} ~& s_{32} ~& \frac{1}{\bar{s}_{33}}\\
\end{pmatrix},\quad
S_2(k)=\begin{pmatrix}
\frac{\bar{s}_{22}}{s_{33}} ~& -\frac{\bar{s}_{21}}{s_{33}} ~& s_{13}  \\[4pt]
-\frac{\bar{s}_{12}}{s_{33}} ~& \frac{\bar{s}_{11}}{s_{33}} ~& s_{23}  \\[4pt]
0 ~& 0 ~& s_{33} \\
\end{pmatrix}.
\eea
\end{proposition}
\begin{proof}
For $n=1,2,$ we can deduce from equation \eqref{2.17} that
\bea
S_n(k)&=&\lim_{x\rightarrow+\infty}\e^{\ii kx\hat{\sigma}}M_n(x,0;k),\nn\\
S_n(k)&=&\e^{\ii kx\hat{\sigma}}[\mu_2^{-1}(x,0;k)M_n(x,0;k)].\nn
\eea
Thus, \eqref{2.9}, integral equation \eqref{2.13} together with the definition \eqref{2.15} of $\infty^n_{jl}$ imply that
\bea
(s^{-1}S_n)_{jl}(k)=\delta_{jl},&&~~ \mbox{if}~\infty^n_{jl}=-\infty,\label{2.19}
\eea
\bea
(S_n)_{jl}(k)=\delta_{jl},&&\quad\quad~ \mbox{if}~\infty^n_{jl}=+\infty.\label{2.20}
\eea
Computing the explicit solution of the algebraic system \eqref{2.19}-\eqref{2.20} and using the relation \eqref{2.12}, we find that $\{S_n(k)\}^2_1$ are exactly given by \eqref{2.18}.
\end{proof}

For convenience, we assume that $s_{33}$ has no zero in $\bfC_-$. Then, we have the following main result in this section.
\begin{theorem}
Let $u_0(x)$ be a function in the Schwartz space $S(\bfR)$ and define the matrix-valued spectral function $s(k)$ via \eqref{2.9}. Define $M(x,t;k)$ as the solution of the following matrix RH problem:

$\bullet$ $M(x,t;k)=\left\{
\begin{aligned}
&M_+(x,t;k)=M_1(x,t;k),\quad k\in\bfC_+=D_1,\\
&M_-(x,t;k)=M_2(x,t;k),\quad k\in\bfC_-=D_2,
\end{aligned}
\right.$ is a sectionally meromorphic function in $\bfC\setminus\bfR$;

$\bullet$ $M(x,t;k)$ satisfies the jump condition
\be\label{2.21}
M_+(x,t;k)=M_-(x,t;k)J(x,t,k),~~k\in\bfR,
\ee
where the matrix $J(x,t,k)$ is defined by
\bea\label{2.22}
\begin{aligned}
J(x,t,k)&=\e^{-\ii(kx-4k^3t)\hat{\sigma}}S^{-1}_2(k)S_1(k)\\
&=\e^{-\ii(kx-4k^3t)\hat{\sigma}}{\left( \begin{array}{ccc}
1 ~& 0 ~& \frac{\bar{s}_{31}}{\bar{s}_{33}} \\[4pt]
0 ~& 1 ~& \frac{\bar{s}_{32}}{\bar{s}_{33}}\\[4pt]
\frac{s_{31}}{s_{33}} ~& \frac{s_{32}}{s_{33}} ~& 1+\frac{\bar{s}_{31}}{\bar{s}_{33}}\frac{s_{31}}{s_{33}}+\frac{\bar{s}_{32}}{\bar{s}_{33}}\frac{s_{32}}{s_{33}}
\end{array}
\right)};
\end{aligned}
\eea

$\bullet$  $M(x,t;k)$ has the following asymptotics:
\be\label{2.23}
M(x,t;k)=I+O\bigg(\frac{1}{k}\bigg),~k\rightarrow\infty,
\ee

Then $M(x,t;k)$ exists and is unique.

Define $\{u(x,t),\bar{u}(x,t)\}$ in terms of $M(x,t;k)$ by
\bea
\bar{u}(x,t)&=&2\ii\lim_{k\rightarrow\infty}(k M(x,t;k))_{31},\label{2.24}\\
u(x,t)&=&2\ii\lim_{k\rightarrow\infty}(k M(x,t;k))_{32}.\label{2.25}
\eea
Then $u(x,t)$ solves the Sasa--Satsuma equation \eqref{1.1}. Furthermore, $u(x,0)=u_0(x).$
\end{theorem}
\begin{proof}
The existence and uniqueness for the solution of above RH problem is a consequence of a `vanishing lemma' for the associated RH problem with the vanishing condition at infinity $M(k)=O(1/k)$, $k\rightarrow\infty$. This fact holds due to the jump matrix $J(x,t,k)$ in \eqref{2.22} is positive definite \cite{MJA}. The proof of a function expressed in terms of the solution of a RH problem with specific dependence on the external parameters $(x,t)$ solves certain nonlinear PDE possessing a Lair pair representation, follows from standard arguments using the dressing method \cite{VEZ1,VEZ2} (see also \cite{XG,JX2}) that if $M$ solves the above RH problem and $u(x,t)$ is defined by \eqref{2.25}, then $u(x,t)$ solves SS equation \eqref{1.1}.

In order to verify the initial condition, one observes that for $t=0$, the RH problem reduces to that associated with $u_0(x)$, which yields $u(x,0)=u_{0}(x)$, owing to the uniqueness of the solution of the RH problem.
\end{proof}

If we rewrite a $3\times3$ matrix $A$ as a block form
$$A=\begin{pmatrix}
A_{11}~& A_{12}\\
A_{21}~& A_{22}\\
\end{pmatrix},$$ where $A_{22}$ is scalar. Then, the above RH problem can be rewritten as the following form:
\be\label{2.26}
\left\{
\begin{aligned}
&M_+(x,t;k)=M_-(x,t;k)J(x,t,k),~~k\in\bfR,\\
&M(x,t;k)\rightarrow I,\qquad\qquad\qquad\qquad\quad k\rightarrow\infty,
\end{aligned}
\right.
\ee
where
\be\label{2.27}
\begin{aligned}
J(x,t,k)&=\begin{pmatrix}
I~& \rho^\dag(\bar{k})\e^{-t\Phi}\\[4pt]
\rho(k)\e^{t\Phi}~& 1+\rho(k)\rho^\dag(\bar{k})
\end{pmatrix},\\
\Phi(\zeta,k)&=2\ii\zeta k-8\ii k^3,~\zeta=\frac{x}{t},~\rho(k)=\bigg(\frac{s_{31}}{s_{33}}~ \frac{s_{32}}{s_{33}}\bigg).
\end{aligned}
\ee
According to the symmetry relation \eqref{2.12}, we can conclude that
\be\label{2.28}
\rho(k)=\overline{\rho(-\bar{k})}\begin{pmatrix}
0 ~& 1\\
1~ & 0
\end{pmatrix}.
\ee
Moreover, we have
\bea
\begin{pmatrix}
\bar{u}(x,t)~& u(x,t)
\end{pmatrix}=2\ii\lim_{k\rightarrow\infty}(k M(x,t;k))_{21}.\label{2.29}
\eea
In the following long-time asymptotic analysis, we will focus on the new block RH problem \eqref{2.26} with the jump matrix defined by \eqref{2.27}.

\section{Long-time asymptotic analysis}
\setcounter{equation}{0}
\setcounter{lemma}{0}
\setcounter{theorem}{0}
An analogue of the classical steepest descent method for RH problems was invented by Deift and Zhou \cite{PD}, we consider the stationary points of the function $\Phi$, that is, taking
$$\frac{\dd\Phi(\zeta,k)}{\dd k}=0,$$
the stationary phase points are obtained for $x>0$ as
\be\label{3.1}
\pm k_0=\pm\sqrt{\frac{\zeta}{12}},
\ee
and the signature table for Re$\Phi(\zeta,k)$ is shown in Fig. \ref{fig1}. Let $M>1$ be a given constant, and let $\mathcal{I}$ denote the interval $\mathcal{I}=(0,M]$. We restrict our attention here to the physically interesting region $\zeta\in\mathcal{I}$.

\begin{figure}[htbp]
  \centering
  \includegraphics[width=3.5in]{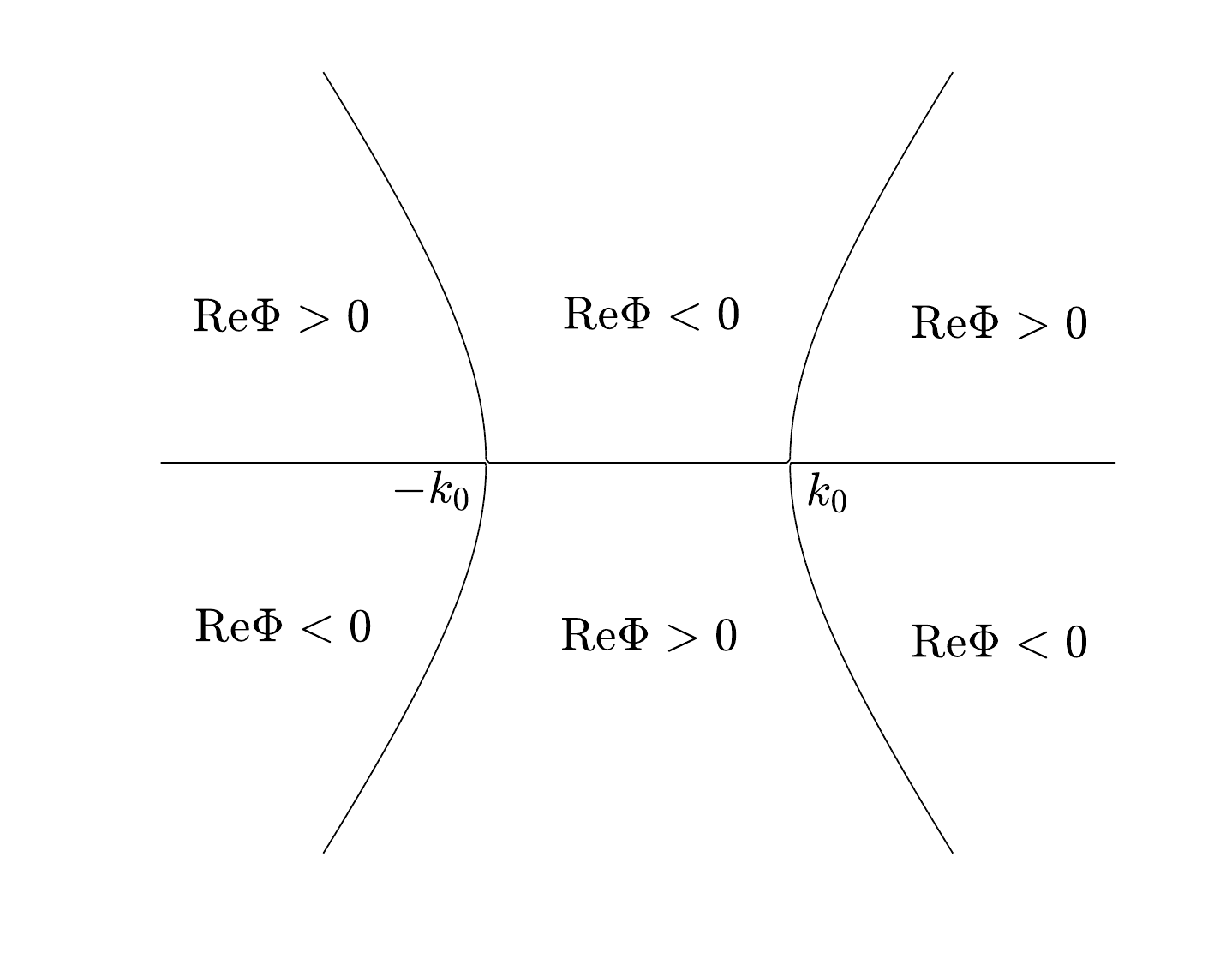}
  \caption{The signature table for Re$\Phi(\zeta,k)$ in the complex $k$-plane.}\label{fig1}
\end{figure}
\subsection{Transformations of the original RH problem}
In this section, we aim to transform the associated original RH problem \eqref{2.26} to a solvable RH problem.
We note that the jump matrix enjoys two distinct factorizations:
\bea\label{3.2}
J=\left\{
\begin{aligned}
&\left( \begin{array}{cc}
I ~& 0\\[4pt]
\rho\e^{t\Phi} ~& 1 \\
\end{array}
\right )
\left( \begin{array}{ccc}
I ~& \rho^\dag\e^{-t\Phi} \\[4pt]
0 ~& 1
\end{array}
\right ),\\
&\left( \begin{array}{cc}
I ~& \frac{\rho^\dag\e^{-t\Phi}}{1+\rho\rho^\dag}\\[4pt]
0 & 1 \\
\end{array}
\right )
\left( \begin{array}{cc}
(I+\rho^\dag\rho)^{-1} ~& 0\\[4pt]
0 ~& 1+\rho\rho^\dag\\
\end{array}
\right )\left( \begin{array}{cc}
I ~& 0\\[4pt]
\frac{\rho\e^{t\Phi}}{1+\rho\rho^\dag} ~& 1 \\
\end{array}
\right ).
\end{aligned}
\right.
\eea
We introduce a function $\delta(k)$ as the solution of the matrix RH problem
\be\label{3.3}
\left\{
\begin{aligned}
\delta_+(k)&=(I+\rho^\dag(k)\rho(k))\delta_-(k),~~|k|<k_0,\\
&=\delta_-(k),\qquad\qquad\qquad~~~~|k|>k_0,\\
\delta(k)&\rightarrow I,\qquad\qquad\qquad\qquad\quad k\rightarrow\infty.
\end{aligned}
\right.
\ee
Since the jump matrix $I+\rho^\dag(k)\rho(k)$ is positive definite, the vanishing lemma \cite{MJA} yields the existence and uniqueness of the function $\delta(k)$. Furthermore, we have
\be\label{3.4}
\left\{
\begin{aligned}
\det\delta_+(k)&=(1+\rho(k)\rho^\dag(k))\det\delta_-(k),~~|k|<k_0,\\
&=\det\delta_-(k),\qquad\quad~\qquad\qquad|k|>k_0,\\
\det\delta(k)&\rightarrow 1,\qquad\qquad\qquad\qquad\qquad~~ k\rightarrow\infty.
\end{aligned}
\right.
\ee
By Plemelj formula, we can obtain
\be\label{3.5}
\det\delta(k)=\exp\bigg\{\frac{1}{2\pi\ii}\int_{-k_0}^{k_0}\frac{\ln(1+\rho(\xi)\rho^\dag(\xi))}{\xi-k}\dd\xi\bigg\}
=\bigg(\frac{k-k_0}{k+k_0}\bigg)^{-\ii\nu}\e^{\chi(k)},
\ee
where
\bea
\nu&=&\frac{1}{2\pi}\ln(1+\rho(k_0)\rho^\dag(k_0))>0,\label{3.6}\\
\chi(k)&=&\frac{1}{2\pi\ii}\int_{-k_0}^{k_0}\ln\bigg(\frac{1+\rho(\xi)\rho^\dag(\xi)}
{1+\rho(k_0)\rho^\dag(k_0)}\bigg)\frac{\dd\xi}{\xi-k}.\label{3.7}
\eea
Here we have used the relation
\be\label{3.8}
1+\rho(k_0)\rho^\dag(k_0)=1+\rho(-k_0)\rho^\dag(-k_0),
\ee
which follows from the second symmetry relation in \eqref{2.12}, more precisely,
$$s_{31}(-k)=\bar{s}_{32}(\bar{k}),~s_{32}(-k)=\bar{s}_{31}(\bar{k}),~s_{33}(-k)=\bar{s}_{33}(\bar{k}).$$

On the other hand, for $|k|<k_0$, it follows from \eqref{3.3} that
\be\label{3.9}
\lim_{\epsilon\rightarrow0^+}\delta(k-\ii\epsilon)=(I+\rho^\dag(k)\rho(k))^{-1}\lim_{\epsilon\rightarrow0^+}\delta(k+\ii\epsilon).
\ee
If we let $g(k)=(\delta^\dag(\bar{k}))^{-1}$, then we get
\bea
g_+(k)&=&\lim_{\epsilon\rightarrow0^+}g(k+\ii\epsilon)=\lim_{\epsilon\rightarrow0^+}(\delta^\dag(\bar{k}-\ii\epsilon))^{-1}\nn\\
&=&(I+\rho^\dag(k)\rho(k))\lim_{\epsilon\rightarrow0^+}(\delta^\dag(\bar{k}+\ii\epsilon))^{-1}\nn\\
&=&(I+\rho^\dag(k)\rho(k))g_-(k).\nn
\eea
Therefore, we find
\be\label{3.10}
(\delta^\dag(\bar{k}))^{-1}=\delta(k).
\ee
Direct calculation as in \cite{XG} and the maximum principle show that
\be\label{3.11}
|\delta(k)|\leq\text{const}<\infty,\quad |\det\delta(k)|\leq\text{const}<\infty,
\ee
for all $k$, where we define $|A|=(\text{tr}A^\dag A)^{1/2}$ for any matrix $A$.

Define
\be\label{3.12}
\Delta(k)=\begin{pmatrix}
\delta^{-1}(k)~& 0\\[4pt]
0 ~& \det\delta(k)
\end{pmatrix}.
\ee
Introduce
$$M^{(1)}(x,t;k)=M(x,t;k)\Delta^{-1}(k),$$ and reverse the orientation for $|k|<k_0$ as shown in Fig. \ref{fig2},
then $M^{(1)}$ satisfies the following RH problem
\be\label{3.13}
\left\{
\begin{aligned}
&M^{(1)}_+(x,t;k)=M^{(1)}_-(x,t;k)J^{(1)}(x,t,k),~k\in \bfR,\\
&M^{(1)}(x,t;k)\rightarrow I,\qquad\qquad\qquad\qquad~~~~~ k\rightarrow\infty,
\end{aligned}
\right.
\ee
where
\bea\label{3.14}
J^{(1)}&=&\Delta_-J\Delta_+^{-1}\nn\\
&=&\left\{
\begin{aligned}
&\left( \begin{array}{cc}
I ~& 0\\[4pt]
\rho_1\delta\det\delta\e^{t\Phi} ~& 1\\
\end{array}
\right )
\left( \begin{array}{cc}
I ~& \frac{\delta^{-1}\rho_4\e^{-t\Phi}}{\det\delta} \\[4pt]
0 ~& 1 \\
\end{array}
\right ),~\qquad |k|>k_0,\\
&\left( \begin{array}{cc}
I ~& 0\\[4pt]
\rho_3\delta_-\det\delta_-\e^{t\Phi} ~& 1 \\
\end{array}
\right )\left( \begin{array}{cc}
I ~& \frac{\delta_+^{-1}\rho_2\e^{-t\Phi}}{\det\delta_+} \\[4pt]
0 ~& 1 \\
\end{array}
\right ),\quad~|k|<k_0,
\end{aligned}
\right.
\eea
and the functions $\{\rho_j(k)\}_1^4$ are defined by
\be\label{3.15}
\begin{aligned}
\rho_1(k)&=\rho(k),\qquad\qquad\qquad\rho_2(k)=-\frac{\rho^\dag(\bar{k})}{1+\rho(k)\rho^\dag(\bar{k})},\\
\rho_3(k)&=-\frac{\rho(k)}{1+\rho(k)\rho^\dag(\bar{k})},\quad \rho_4(k)=\rho^\dag(\bar{k}).
\end{aligned}
\ee

\begin{figure}[htbp]
  \centering
  \includegraphics[width=3in]{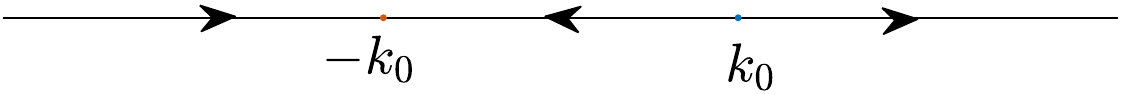}
  \caption{The oriented contour on $\bfR$.}\label{fig2}
\end{figure}

Our next goal is to deform the contour. However, the spectral functions $\{\rho_j(k)\}_1^4$ have limited domains of analyticity,
we will follow the idea of \cite{PD,JL3} and decompose each of the functions $\{\rho_j(k)\}_1^4$ into an analytic part and a small remainder. The analytic part of the jump matrix will be deformed, whereas the small remainder will be left on the original contour. We introduce the open subsets $\{\Omega_j\}_1^4$, as displayed in Fig. \ref{fig3} such that
\bea
\Omega_1\cup\Omega_3&=&\{k|\text{Re}\Phi(\zeta,k)<0\},\nn\\
\Omega_2\cup\Omega_4&=&\{k|\text{Re}\Phi(\zeta,k)>0\}.\nn
\eea
\begin{figure}[htbp]
  \centering
  \includegraphics[width=3.5in]{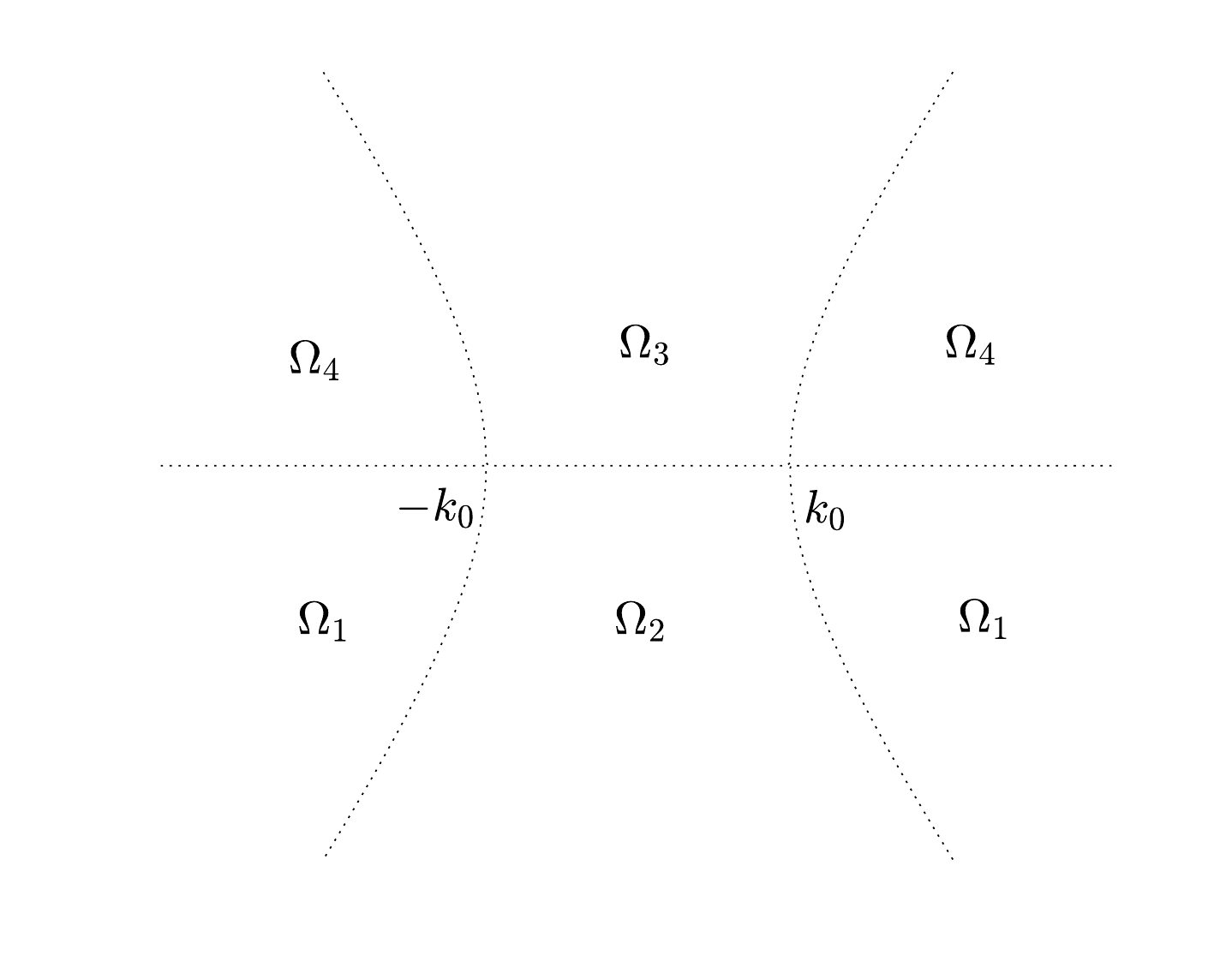}
  \caption{The open sets $\{\Omega_j\}_1^4$ in the complex $k$-plane.}\label{fig3}
\end{figure}

In fact, we have the following lemma.
\begin{lemma}\label{lem3.1}
There exist decompositions
\be\label{3.16}
\rho_j(k)=\left\{
\begin{aligned}
&\rho_{j,a}(x,t,k)+\rho_{j,r}(x,t,k),\quad |k|>k_0,~k\in\bfR,~j=1,4,\\
&\rho_{j,a}(x,t,k)+\rho_{j,r}(x,t,k),\quad |k|<k_0,~k\in\bfR,~j=2,3,
\end{aligned}
\right.
\ee
where the functions $\{\rho_{j,a}, \rho_{j,r}\}^4_1$ have the following properties:

(1) For each $\zeta\in\mathcal{I}$ and each $t>0$, $\rho_{j,a}(x,t,k)$ is defined and continuous for $k\in\bar{\Omega}_j$ and
analytic for $\Omega_j$, $j= 1,2$.

(2) For $K>0$, the functions $\rho_{j,a}$ for $j=1,2,3,4$ satisfy
\be\label{3.17}
|\rho_{j,a}(x,t,k)-\rho_j(k_0)|\leq C_K|k-k_0|\e^{\frac{t}{4}|\text{Re}\Phi(\zeta,k)|},~k\in\bar{\Omega}_j,
~|k|<K,~ t>0,~\zeta\in\mathcal{I}.
\ee
In particular, the functions $\rho_{1,a}$ and $\rho_{4,a}$ satisfy
\be\label{3.18}
|\rho_{j,a}(x,t,k)|\leq \frac{C}{1+|k|^2}\e^{\frac{t}{4}|\text{Re}\Phi(\zeta,k)|},
~t>0,~k\in\bar{\Omega}_j,~\zeta\in\mathcal{I},~j=1,4,
\ee
where the constants $C_K,C$ are independent of $\zeta, k, t$, but $C_K$ may depend on $K$.

(3) The $L^1, L^2$ and $L^\infty$ norms of the functions $\rho_{1,r}(x,t,\cdot)$ and $\rho_{4,r}(x,t,\cdot)$ on $(-\infty,-k_0)\cup(k_0,\infty)$ are $O(t^{-3/2})$ as $t\rightarrow\infty$ uniformly with respect to $\zeta\in\mathcal{I}$.

(4) The $L^1, L^2$ and $L^\infty$ norms of the functions $\rho_{2,r}(x,t,\cdot)$ and $\rho_{3,r}(x,t,\cdot)$ on $(-k_0,k_0)$ are $O(t^{-3/2})$ as $t\rightarrow\infty$ uniformly with respect to $\zeta\in\mathcal{I}$.

(5) For $j=1,2,3,4$, the following symmetries hold:
\be\label{3.19}
\rho_{j,a}(x,t,k)=\overline{\rho_{j,a}(x,t,-\bar{k})}\begin{pmatrix}
0~& 1\\
1~& 0
\end{pmatrix},~\rho_{j,r}(x,t,k)=\overline{\rho_{j,r}(x,t,-\bar{k})}\begin{pmatrix}
0~& 1\\
1~& 0
\end{pmatrix}.
\ee
\end{lemma}
\begin{proof}
We first consider the decomposition of $\rho_1(k)$. Denote $\Om_1=\Om_1^+\cup\Om_1^-$, where $\Om_1^+$ and $\Om_1^-$ denote the parts of $\Om_1$ in the right and left half-planes, respectively. We derive a decomposition of $\rho_1(k)$ in $\Om_1^+$, and then extend it to $\Om_1^-$ by symmetry. Since $\rho_1(k)\in C^\infty(\bfR)$, then for $n=0,1,2$,
\bea
\rho_1^{(n)}(k)&=&\frac{\dd^n}{\dd k^n}\bigg(\sum_{j=0}^6\frac{\rho_1^{(j)}(k_0)}{j!}(k-k_0)^j\bigg)+O((k-k_0)^{7-n}),~~ k\rightarrow k_0,\label{3.20}
\eea
Let
\be\label{3.21}
f_0(k)=\sum_{j=4}^{10}\frac{a_j}{(k-\ii)^j},
\ee
where $\{a_j\}_4^{10}$ are complex constants such that
\be\label{3.22}
f_0(k)=\sum_{j=0}^6\frac{\rho_1^{(j)}(k_0)}{j!}(k-k_0)^j+O((k-k_0)^7),~~k\rightarrow k_0.
\ee
It is easy to verify that \eqref{3.22} imposes seven linearly independent conditions on the $a_j$, hence the coefficients $a_j$ exist and are unique. Letting $f=\rho_1-f_0$, it follows that \\
(i) $f_0(k)$ is a rational function of $k\in\bfC$ with no poles in $\Om_1^+$;\\
(ii) $f_0(k)$ coincides with $\rho_1(k)$ to six order at $k_0$, more precisely,
\be\label{3.23}
\frac{\dd^n}{\dd k^n}f(k)=\left\{
\begin{aligned}
&O((k-k_0)^{7-n}),~ k\rightarrow k_0,\\
&O(k^{-4-n}),~~~~~\quad k\rightarrow\infty,
\end{aligned}
\quad k\in\bfR,~n=0,1,2.
\right.
\ee
The decomposition of $\rho_1(k)$ can be derived as follows. The map $k\mapsto\phi=\phi(k)$ defined by $\phi(k)=-\ii\Phi(\zeta,k)$ is a bijection $[k_0,\infty)\mapsto(-\infty,16k_0^3]$, so we may define a function $F$ by
\be\label{3.24}
F(\phi)=\left\{
\begin{aligned}
&\frac{(k-\ii)^3}{k-k_0}f(k),~\phi<16k_0^3,\\
&0,\qquad\qquad\quad \phi\geq16k_0^3.
\end{aligned}
\right.
\ee
Then,
\berr
F^{(n)}(\phi)=\bigg(\frac{1}{24(k^2_0-k^2)}\frac{\partial}{\partial k}\bigg)^n\bigg(\frac{(k-\ii)^3}{k-k_0}f(k)\bigg),~~\phi<16k_0^3.
\eerr
By \eqref{3.23}, $F\in C^1(\bfR)$ and $F^{(n)}(\phi)=O(|\phi|^{-2/3})$ as $|\phi|\rightarrow\infty$ for $n=0,1,2$. In particular,
\be\label{3.25}
\bigg\|\frac{\dd^nF}{\dd\phi^n}\bigg\|_{L^2(\bfR)}<\infty,\quad n=0,1,2,
\ee
that is, $F^T$ belongs to $H^2(\bfR)\times H^2(\bfR)$, where the superscript $`T$' denote matrix transpose. By the Fourier transform $\hat{F}(s)$ defined by
\be\label{3.26}
\hat{F}(s)=\frac{1}{2\pi}\int_\bfR F(\phi)\e^{-\ii\phi s}\dd\phi
\ee
where
\be\label{3.27}
F(\phi)=\int_\bfR\hat{F}(s)\e^{\ii\phi s}\dd s,
\ee
it follows from Plancherel theorem that $\|s^2\hat{F}(s)\|_{L^2(\bfR)}<\infty$. Equations \eqref{3.24} and \eqref{3.27} imply
\be\label{3.28}
f(k)=\frac{k-k_0}{(k-\ii)^3}\int_\bfR\hat{F}(s)\e^{\ii\phi s}\dd s,\quad k>k_0.
\ee
Writing
$$f(k)=f_a(t,k)+f_r(t,k),\quad t>0,~k>k_0,$$
where the functions $f_a$ and $f_r$ are defined by
\bea
f_a(t,k)&=&\frac{k-k_0}{(k-\ii)^3}\int_{-\frac{t}{4}}^\infty\hat{F}(s)\e^{s\Phi(\zeta,k)}\dd s,~\quad t>0,~k\in\Om^+_1,\label{3.29}\\
f_r(t,k)&=&\frac{k-k_0}{(k-\ii)^3}\int^{-\frac{t}{4}}_{-\infty}\hat{F}(s)\e^{s\Phi(\zeta,k)}\dd s,\quad t>0,~k>k_0,\label{3.30}
\eea
we infer that $f_a(t,\cdot)$ is continuous in $\bar{\Om}^+_1$ and analytic in $\Om^+_1$. Moreover, we can get
\bea\label{3.31}
|f_a(t,k)|&\leq&\frac{|k-k_0|}{|k-\ii|^3}\|\hat{F}(s)\|_{L^1(\bfR)}\sup_{s\geq-\frac{t}{4}}\e^{s\text{Re}\Phi(\zeta,k)}\nn\\
&\leq&\frac{C|k-k_0|}{|k-\ii|^3}\e^{\frac{t}{4}|\text{Re}\Phi(\zeta,k)|},\quad t>0,~k\in\bar{\Om}^+_1,~\zeta\in\mathcal{I}.
\eea
Furthermore, we have
\bea\label{3.32}
|f_r(t,k)|&\leq&\frac{|k-k_0|}{|k-\ii|^3}\int_{-\infty}^{-\frac{t}{4}}s^2|\hat{F}(s)|s^{-2}\dd s\nn\\
&\leq&\frac{C}{1+|k|^2}\|s^2\hat{F}(s)\|_{L^2(\bfR)}\sqrt{\int_{-\infty}^{-\frac{t}{4}}s^{-4}\dd s},\\
&\leq&\frac{C}{1+|k|^2}t^{-3/2},\quad t>0,~k>k_0,~\zeta\in\mathcal{I}.\nn
\eea
Hence, the $L^1,L^2$ and $L^\infty$ norms of $f_r$ on $(k_0,\infty)$ are $O(t^{-3/2})$. Letting
\bea
\rho_{1,a}(t,k)&=&f_0(k)+f_a(t,k),~\quad t>0,~ k\in\bar{\Om}_1^+,\label{3.33}\\
\rho_{1,r}(t,k)&=&f_r(t,k),~~\qquad\qquad t>0,~ k>k_0.\label{3.34}
\eea
For $k<-k_0$, we use the symmetry \eqref{3.19} to extend this decomposition.  Thus, we find a decomposition of $\rho_1$ for $|k|>k_0$ with the properties listed in the statement of the lemma.

The decomposition of $\rho_3$ for the case $|k|<k_0$ is similar, detailed proof can be found in \cite{PD,JL3}, we will be omitted. The decompositions of $\rho_2$ and $\rho_4$ can be obtained from $\rho_3$ and $\rho_1$ by Schwartz conjugation.
\end{proof}

The next transformation is to deform the contour so that the jump matrix involves the exponential factor $\e^{-t\Phi}$ on the parts of the contour where Re$\Phi$ is positive and the factor $\e^{t\Phi}$ on the parts where Re$\Phi$ is negative. More precisely, we put
\be\label{3.35}
M^{(2)}(x,t;k)=M^{(1)}(x,t;k)G(k),
\ee
where
\be\label{3.36}
G(k)=\left\{
\begin{aligned}
&\begin{pmatrix}
I ~& 0\\[4pt]
\rho_{1,a}\delta\det\delta\e^{t\Phi} ~& 1
\end{pmatrix},~~~ k\in U_1,\\
&\begin{pmatrix}
I ~& -\frac{\delta^{-1}\rho_{2,a}\e^{-t\Phi}}{\det\delta}\\[4pt]
0 ~& 1
\end{pmatrix},~~\quad k\in U_2,\\
&\begin{pmatrix}
I ~& 0\\[4pt]
\rho_{3,a}\delta\det\delta\e^{t\Phi} ~& 1
\end{pmatrix},\quad~ k\in U_3,\\
&\begin{pmatrix}
I &~ -\frac{\delta^{-1}\rho_{4,a}\e^{-t\Phi}}{\det\delta}\\[4pt]
0 ~& 1
\end{pmatrix},\quad k\in U_4,\\
&I,~~~~~~~\qquad\qquad\qquad\quad k\in U_5\cup U_6.
\end{aligned}
\right.
\ee
Then the matrix $M^{(2)}(x,t;k)$ satisfies the following RH problem
\be\label{3.37}
M_+^{(2)}(x,t;k)=M_-^{(2)}(x,t;k)J^{(2)}(x,t,k),\quad k\in\Sigma,
\ee
with the jump matrix $J^{(2)}=G_-^{-1}(k)J^{(1)}G_+(k)$ is given by
\bea
J^{(2)}_1&=&\begin{pmatrix}
I ~& 0\\[4pt]
-\rho_{1,a}\delta\det\delta\e^{t\Phi} ~& 1
\end{pmatrix},\quad J^{(2)}_2=\begin{pmatrix}
I ~& -\frac{\delta^{-1}\rho_{2,a}\e^{-t\Phi}}{\det\delta}\\[4pt]
0 ~& 1
\end{pmatrix},\nn\\
 J^{(2)}_3&=&\begin{pmatrix}
I ~& 0\\[4pt]
-\rho_{3,a}\delta\det\delta\e^{t\Phi} ~& 1
\end{pmatrix},\quad
J^{(2)}_4=\begin{pmatrix}
I &~ -\frac{\delta^{-1}\rho_{4,a}\e^{-t\Phi}}{\det\delta}\\[4pt]
0 ~& 1
\end{pmatrix},\label{3.38}\\
J^{(2)}_5&=&\left( \begin{array}{cc}
I ~& 0\\[4pt]
\rho_{1,r}\delta\det\delta\e^{t\Phi} ~& 1\\
\end{array}
\right )
\left( \begin{array}{cc}
I ~& \frac{\delta^{-1}\rho_{4,r}\e^{-t\Phi}}{\det\delta} \\[4pt]
0 ~& 1 \\
\end{array}
\right ),\nn\\
J^{(2)}_6&=&\left( \begin{array}{cc}
I ~& 0\\[4pt]
\rho_{3,r}\delta_-\det\delta_-\e^{t\Phi} ~& 1 \\
\end{array}
\right )\left( \begin{array}{cc}
I ~& \frac{\delta_+^{-1}\rho_{2,r}\e^{-t\Phi}}{\det\delta_+} \\[4pt]
0 ~& 1 \\
\end{array}
\right )\nn
\eea
with $J^{(2)}_i$ denoting the restriction of $J^{(2)}$ to the contour labeled by $i$ in Fig. \ref{fig4}.

\begin{figure}[htbp]
  \centering
  \includegraphics[width=3.5in]{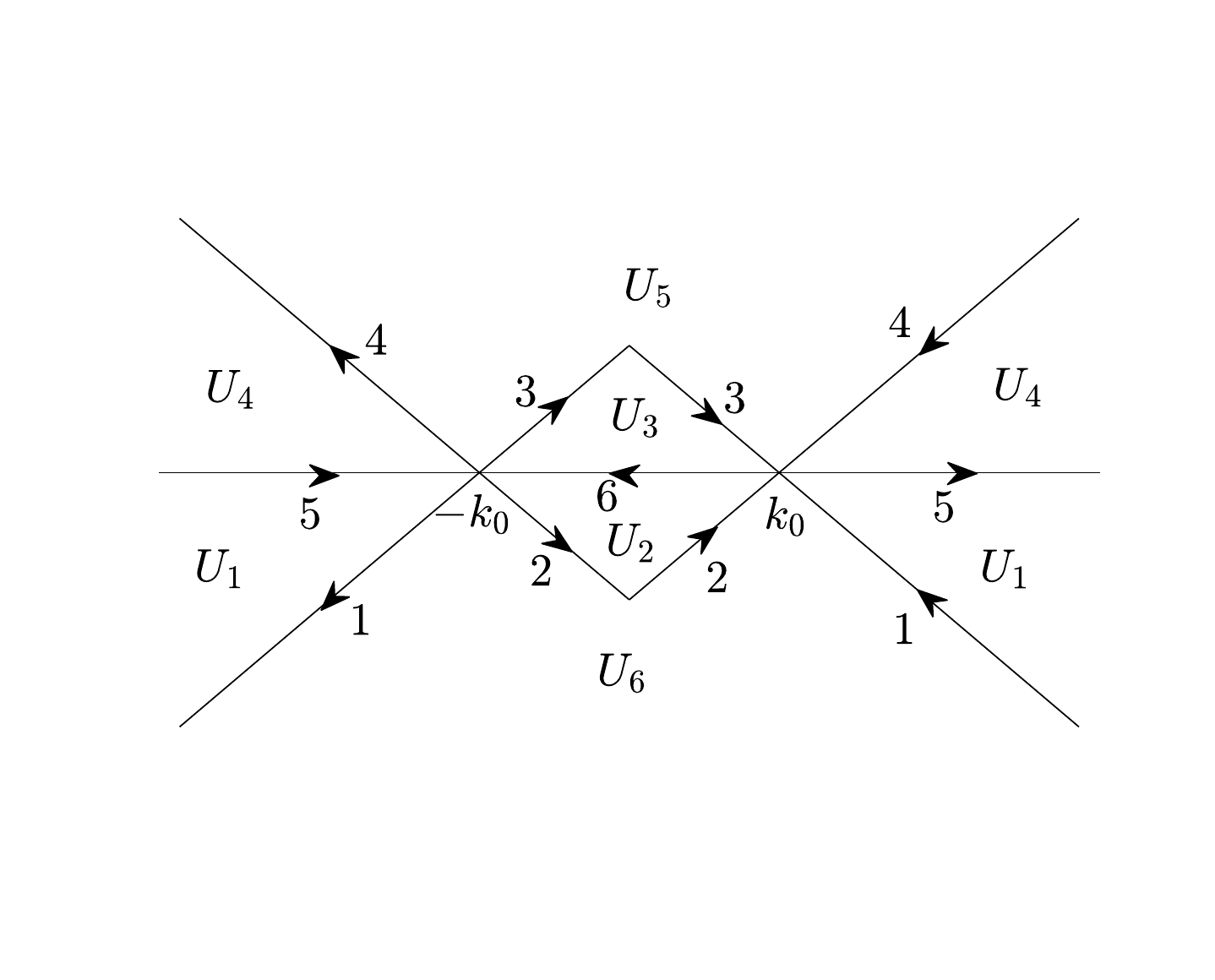}
  \caption{The jump contour $\Sigma$ and the open sets $\{U_j\}_1^6$ in the complex $k$-plane.}\label{fig4}
\end{figure}

Obviously, the jump matrix $J^{(2)}$ decays to identity matrix $I$ as $t\rightarrow\infty$ everywhere except near the critical points $\pm k_0$. This implies that the main contribution to the long-time asymptotics should come from the neighborhoods of the critical points $\pm k_0$.

\subsection{Local models near the stationary phase points $\pm k_0$}
To focus on $-k_0$ and $k_0$, we introduce the following scaling operators
\bea\label{3.39}
\begin{aligned}
S_{-k_0}:&k\mapsto\frac{z}{\sqrt{48tk_0}}-k_0,\\
S_{k_0}:&k\mapsto\frac{z}{\sqrt{48tk_0}}+k_0.
\end{aligned}
\eea
Let $D_\varepsilon(\pm k_0)$ denote the open disk of radius $\varepsilon$ centered at $\pm k_0$ for a small $\varepsilon>0$. Then, the operator $S_{\pm k_0}$ is a bijection from $D_\varepsilon(\pm k_0)$ to the open disk of radius $\sqrt{48tk_0}\varepsilon$ and centered at the origin for all $\zeta\in\mathcal{I}$.

Since the function $\delta(k)$ satisfying a $2\times2$ matrix RH problem \eqref{3.3} can not be solved explicitly, to proceed the next step, we will follow the idea developed in \cite{XG} to use the available function $\det\delta(k)$ to approximate $\delta(k)$ by error estimate. More precisely, we write
\bea
\bigg(\rho_{j,a}\delta\det\delta\e^{t\Phi}\bigg)(k)=\bigg(\rho_{j,a}(\det\delta)^2\e^{t\Phi}\bigg)(k)+
\bigg(\rho_{j,a}(\delta-\det\delta I)\det\delta\e^{t\Phi}\bigg)(k),\label{3.40}
\eea
for $j=1,3$. For the first part in the right-hand side of \eqref{3.40}, we have
\bea
S_{k_0}\big(\rho_{j,a}(\det\delta)^2\e^{t\Phi}\big)(z)&=&\eta^2\varrho^2\rho_{j,a}\bigg(\frac{z}{\sqrt{48tk_0}}+k_0\bigg),\label{3.41}
\eea
where
\be\label{3.42}
\begin{aligned}
\eta&=(192tk_0^3)^\frac{\ii\nu}{2}\e^{8\ii tk_0^3+\chi(k_0)},\\
\varrho&=z^{-\ii\nu}\e^{-\frac{\ii z^2}{4}(1+ z/\sqrt{432tk_0^3})}
\bigg(\frac{2k_0}{z/\sqrt{48tk_0}+2k_0}\bigg)^{-\ii\nu}\e^{(\chi([z/\sqrt{48tk_0}]+k_0)-\chi(k_0))}.
\end{aligned}
\ee
Let $\tilde{\delta}(k)=\rho_{j,a}(k)(\delta(k)-\det\delta(k) I)\e^{t\Phi(\zeta,k)},$ then we have the following estimate.
\begin{lemma}\label{lem3.2}
For $z\in\hat{L}=\{z=\alpha\e^{\frac{3\ii\pi}{4}}:-\infty<\alpha<+\infty\}$, as $t\rightarrow\infty$, the following estimate for $\tilde{\delta}(k)$ hold:
\be
|(S_{k_0}\tilde{\delta})(z)|\leq Ct^{-1/2},\label{3.43}
\ee
where the constant $C>0$ independent of $z,t$.
\end{lemma}
\begin{proof}
It follows from \eqref{3.3} and \eqref{3.4} that $\tilde{\delta}$ satisfies the following RH problem
\be\label{3.44}
\left\{
\begin{aligned}
\tilde{\delta}_+(k)&=(1+\rho(k)\rho^\dag(k))\tilde{\delta}_-(k)+\e^{t\Phi(\zeta,k)}f(k),~~|k|<k_0,\\
&=\tilde{\delta}_-(k),\qquad\qquad~~~~\qquad\qquad\qquad\qquad|k|>k_0,\\
\tilde{\delta}(k)&\rightarrow 0,\qquad\qquad\qquad\qquad\qquad\qquad\qquad\quad k\rightarrow\infty,
\end{aligned}
\right.
\ee
where $f(k)=[\rho_{j,a}(\rho^\dag\rho-\rho\rho^\dag I)\delta_-](k)$. Then the function $\tilde{\delta}(k)$ can be expressed by
\bea\label{3.45}
\begin{aligned}
\tilde{\delta}(k)&=X(k)\int_{-k_0}^{k_0}\frac{\e^{t\Phi(\zeta,\xi)}f(\xi)}{X_+(\xi)(\xi-k)}\dd\xi,\\
X(k)&=\exp\bigg\{\frac{1}{2\pi\ii}\int_{-k_0}^{k_0}\frac{\ln(1+\rho(\xi)\rho^\dag(\xi))}{\xi-k}\dd\xi\bigg\}.
\end{aligned}
\eea
On the other hand, we find that $\rho_{j,a}(\rho^\dag\rho-\rho\rho^\dag I)=\rho_{j,r}(\rho\rho^\dag I-\rho^\dag\rho)$. Thus, we get $f(k)=O(t^{-3/2})$. Similar to Lemma \ref{lem3.1}, $f(k)$ can be decomposed into two parts: $f_a(x,t,k)$ which has an analytic continuation to $\Om_3$ and $f_r(x,t,k)$ which is a small remainder. In particular,
\bea\label{3.46}
\begin{aligned}
|f_a(x,t,k)|&\leq\frac{C}{1+|k|^2} \e^{\frac{t}{4}|\text{Re}\Phi(\zeta,k)|},~k\in L_t,\\
|f_r(x,t,k|&\leq\frac{C}{1+|k|^2} t^{-3/2},\qquad~~k\in(-k_0,k_0),
\end{aligned}
\eea
where $$L_t=\{k=k_0-\frac{1}{t}+\alpha\e^{\frac{3\ii\pi}{4}}:0\leq\alpha\leq\sqrt{2}(k_0-\frac{1}{2t})\}
\cup\{k=-k_0+\alpha\e^{\frac{\ii\pi}{4}}:0\leq\alpha\leq\sqrt{2}(k_0-\frac{1}{2t})\}.$$
Therefore, for $z\in\hat{L}$, we can find
\bea
(S_{k_0}\tilde{\delta})(z)&=&X(k_0+\frac{z}{\sqrt{48tk_0}})\int_{k_0-\frac{1}{t}}^{k_0}
\frac{\e^{t\Phi(\zeta,\xi)}f(\xi)}{X_+(\xi)(\xi-k_0-\frac{z}{\sqrt{48tk_0}})}\dd\xi\nn\\
&&+X(k_0+\frac{z}{\sqrt{48tk_0}})\int_{-k_0}^{k_0-\frac{1}{t}}
\frac{\e^{t\Phi(\zeta,\xi)}f_a(x,t,\xi)}{X_+(\xi)(\xi-k_0-\frac{z}{\sqrt{48tk_0}})}\dd\xi\nn\\
&&+X(k_0+\frac{z}{\sqrt{48tk_0}})\int_{-k_0}^{k_0-\frac{1}{t}}
\frac{\e^{t\Phi(\zeta,\xi)}f_r(x,t,\xi)}{X_+(\xi)(\xi-k_0-\frac{z}{\sqrt{48tk_0}})}\dd\xi\nn\\
&=&I_1+I_2+I_3.\nn
\eea
Then,
\bea
|I_1|&\leq&\int_{k_0-\frac{1}{t}}^{k_0}\frac{|f(\xi)|}{|\xi-k_0-\frac{z}{\sqrt{48tk_0}}|}\dd\xi\leq Ct^{-3/2}\bigg|\ln\bigg|1-\frac{4\sqrt{3}}{zt^{1/2}}\bigg|\bigg|\leq Ct^{-2},\nn\\
|I_3|&\leq&\int_{-k_0}^{k_0-\frac{1}{t}}\frac{|f_r(x,t,\xi)|}{|\xi-k_0-\frac{z}{\sqrt{48tk_0}}|}\dd\xi\leq Ct^{-1/2}.\nn
\eea
By the Cauchy's theorem, we can evaluate $I_2$ along the contour $L_t$ instead of the interval $(-k_0,k_0-\frac{1}{t})$.
Using the fact $\text{Re}\Phi(\zeta,k)<0$ in $\Om_3$, we can obtain $|I_2|\leq C\e^{-ct}$. This completes the proof of the lemma.
\end{proof}
\begin{remark}
There is a similar estimate
\be\label{3.47}
|(S_{k_0}\hat{\delta})(z)|\leq Ct^{-1/2},~t\rightarrow\infty,~
\ee
for $z\in\bar{\hat{L}}$, where $\hat{\delta}(k)=[\delta^{-1}(k)-(\det\delta)^{-1}I]\rho_{j,a}(k)\e^{-t\Phi(\zeta,k)}$, $j=2,4$.
\end{remark}

In other words, we have the following important relation:
\be\label{3.48}
(S_{k_0}J^{(2)}_i)(x,t,z)=\tilde{J}(x,t,z)+O(t^{-1/2}),~i=1,\ldots,4,
\ee
where $\tilde{J}(x,t,z)$ is given by
\be\label{3.49}
\tilde{J}(x,t,z)=\left\{
\begin{aligned}
&\begin{pmatrix}
I ~& 0\\[4pt]
-\eta^2\varrho^2(t,z)\rho_{1,a}\big(\frac{z}{\sqrt{48tk_0}}+k_0\big) ~& 1
\end{pmatrix},\qquad k\in (\mathcal{X}_{k_0}^\varepsilon)_1,\\
&\begin{pmatrix}
I ~& -\eta^{-2}\varrho^{-2}(t,z)\rho_{2,a}\big(\frac{z}{\sqrt{48tk_0}}+k_0\big)\\[4pt]
0 ~& 1
\end{pmatrix},\quad k\in (\mathcal{X}_{k_0}^\varepsilon)_2,\\
&\begin{pmatrix}
I ~& 0\\[4pt]
-\eta^2\varrho^2(t,z)\rho_{3,a}\big(\frac{z}{\sqrt{48tk_0}}+k_0\big) ~& 1
\end{pmatrix},\qquad~ k\in (\mathcal{X}_{k_0}^\varepsilon)_3,\\
&\begin{pmatrix}
I ~& -\eta^{-2}\varrho^{-2}(t,z)\rho_{4,a}\big(\frac{z}{\sqrt{48tk_0}}+k_0\big)\\[4pt]
0 ~& 1
\end{pmatrix},\quad~ k\in (\mathcal{X}_{k_0}^\varepsilon)_4.
\end{aligned}
\right.
\ee
Here $\mathcal{X}_{k_0}=X+k_0$ denote the cross $X$ defined by \eqref{3.50} centered at $k_0$, $\mathcal{X}_{k_0}^\varepsilon=\mathcal{X}_{k_0}\cap D_\varepsilon(k_0)$, and $X=X_1\cup X_2\cup X_3\cup X_4\subset\bfC$ be the cross defined by
\be\label{3.50}
\begin{aligned}
X_1&=\{l\e^{-\frac{\ii\pi}{4}}|0\leq l<\infty\},~~ X_2=\{l\e^{-\frac{3\ii\pi}{4}}|0\leq l<\infty\},\\
X_3&=\{l\e^{\frac{3\ii\pi}{4}}|0\leq l<\infty\},~~~X_4=\{l\e^{\frac{\ii\pi}{4}}|0\leq l<\infty\},
\end{aligned}
\ee
and oriented as in Fig. \ref{fig5}.
\begin{figure}[htbp]
  \centering
  \includegraphics[width=2.8in]{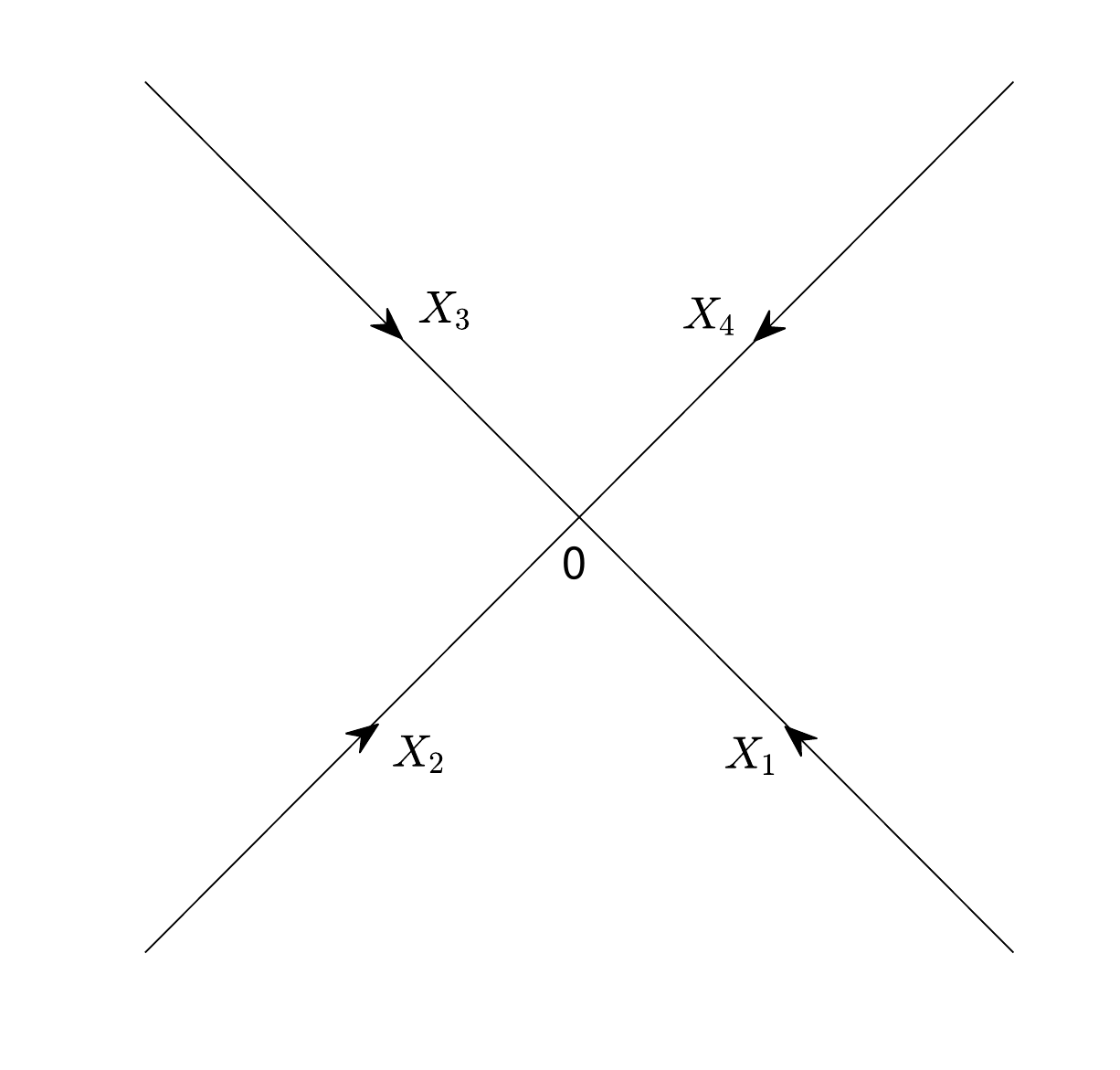}
  \caption{The contour $X=X_1\cup X_2\cup X_3\cup X_4$.}\label{fig5}
\end{figure}

For the jump matrix $\tilde{J}(x,t,z)$, for any fixed $z\in X$ and $\zeta\in\mathcal{I}$, we have
\be\label{3.51}
\begin{aligned}
\rho_{1,a}\big(\frac{z}{\sqrt{48tk_0}}+k_0\big)&\rightarrow \rho(k_0),\\
\rho_{3,a}\big(\frac{z}{\sqrt{48tk_0}}+k_0\big)&\rightarrow-\frac{\rho(k_0)}{1+\rho(k_0)\rho^\dag(k_0)},\\ \varrho^2\rightarrow\e^{-\frac{\ii z^2}{2}}z^{-2\ii\nu},
\end{aligned}
\ee
as $t\rightarrow\infty$.
This implies that the jump matrix $\tilde{J}$ tend to the matrix $J^{(k_0)}$
\be\label{3.52}
J^{(k_0)}(x,t,z)=\left\{
\begin{aligned}
&\begin{pmatrix}
I ~& 0\\[4pt]
-\eta^2\e^{-\frac{\ii z^2}{2}}z^{-2\ii\nu}\rho(k_0) ~& 1
\end{pmatrix},~~\qquad z\in X_1,\\
&\begin{pmatrix}
I ~& \eta^{-2}\e^{\frac{\ii z^2}{2}}z^{2\ii\nu}\frac{\rho^\dag(k_0)}{1+\rho(k_0)\rho^\dag(k_0)}\\[4pt]
0 ~& 1
\end{pmatrix},~~\quad z\in X_2,\\
&\begin{pmatrix}
I ~& 0\\[4pt]
\eta^2\e^{-\frac{\ii z^2}{2}}z^{-2\ii\nu}\frac{\rho(k_0)}{1+\rho(k_0)\rho^\dag(k_0)} ~& 1
\end{pmatrix},\quad z\in X_3,\\
&\begin{pmatrix}
I ~& -\eta^{-2}\e^{\frac{\ii z^2}{2}}z^{2\ii\nu}\rho^\dag(k_0)\\[4pt]
0 ~& 1
\end{pmatrix},\qquad\quad~ z\in X_4
\end{aligned}
\right.
\ee
for large $t$.
\begin{theorem}\label{th3.1}
The following RH problem:
\be\label{3.53}\left\{
\begin{aligned}
&M^X_+(x,t;z)=M^X_-(x,t;z)J^X(x,t,z),~\text{for~almost~every}~z\in X,\\
&M^X(x,t;z)\rightarrow I,~~~~~\quad\qquad\qquad\qquad\text{as}~z\rightarrow\infty,
\end{aligned}
\right.
\ee
with the jump matrix $J^X(x,t,z)=\eta^{\hat{\sigma}}J^{(k_0)}(x,t,z)$ has a unique solution $M^X(x,t;z)$. If we let
\be\label{3.54}
M^X(x,t;z)=I+\frac{M^X_1(x,t)}{z}+O\bigg(\frac{1}{z^2}\bigg),\quad z\rightarrow\infty,
\ee
then we have the following relation
\be\label{3.55}
(M^X_1)_{21}=-\ii\beta^X,~\beta^X=\frac{\nu\Gamma(-\ii\nu)\e^{\frac{\ii\pi}{4}-\frac{\pi\nu}{2}}}{\sqrt{2\pi}}\rho(k_0),
\ee
where $\Gamma(\cdot)$ denotes the standard Gamma function. Moreover, for each compact subset $\mathcal{D}$ of $\bfC$,
\be\label{3.56}
\sup_{\rho(k_0)\in\mathcal{D}}\sup_{z\in\bfC\setminus X}|M^X(x,t;z)|<\infty.
\ee
\end{theorem}
\begin{proof}
See appendix.
\end{proof}
Therefore, we can approximate $M^{(2)}$ in the neighborhood $D_\varepsilon(k_0)$ of $k_0$ by
\be\label{3.57}
M^{(k_0)}(x,t;k)=\eta^{-\hat{\sigma}}M^X(x,t;z).
\ee
\begin{lemma}\label{lem3.3}
For each $\zeta\in\mathcal{I}$ and $t>0$, the function $M^{(k_0)}(x,t;k)$ defined in \eqref{3.57} is an analytic function of $k\in D_\varepsilon(k_0)\setminus\mathcal{X}_{k_0}^\varepsilon$.
Across $\mathcal{X}_{k_0}^\varepsilon$, $M^{(k_0)}$ satisfied the jump condition $M_+^{(k_0)}=M_-^{(k_0)}J^{(k_0)}$, where the  jump matrix $J^{(k_0)}$ satisfies the following estimates for $1\leq p\leq\infty$:
\be\label{3.59}
\|J^{(2)}-J^{(k_0)}\|_{L^p(\mathcal{X}_{k_0}^\varepsilon)}\leq Ct^{-\frac{1}{2}-\frac{1}{2p}}\ln t,
\quad t>3,~\zeta\in\mathcal{I},
\ee
where $C>0$ is a constant independent of $t,\zeta,z$. Moreover, as $t\rightarrow\infty$,
\be\label{3.60}
\|(M^{(k_0)})^{-1}(x,t;k)-I\|_{L^\infty(\partial D_\varepsilon(k_0))}=O(t^{-1/2}),
\ee
and
\be\label{3.61}
-\frac{1}{2\pi\ii}\int_{\partial D_\varepsilon(k_0)}((M^{(k_0)})^{-1}(x,t;k)-I)\dd k=\frac{\eta^{-\hat{\sigma}}M^X_1}
{\sqrt{48tk_0}}+O(t^{-1}).
\ee
\end{lemma}
\begin{proof}
The analyticity of $M^{(k_0)}$ is obvious.
On the other hand, we have $J^{(2)}-J^{(k_0)}=J^{(2)}-\tilde{J}+\tilde{J}-J^{(k_0)}$. However, a careful computation as the Lemma 3.35 in \cite{PD}, we conclude that
\be\label{3.63}
\|\tilde{J}-J^{(k_0)}\|_{L^\infty((\mathcal{X}_{k_0}^\varepsilon)_4)}\leq C|\e^{\frac{\ii\gamma}{2}z^2}|t^{-1/2}\ln t,\quad 0<\gamma<\frac{1}{2},~t>3,~\zeta\in\mathcal{I},
\ee
for $k\in(\mathcal{X}_{k_0}^\varepsilon)_4$, that is, $z=\sqrt{48tk_0}u\e^{\frac{\ii\pi}{4}}$, $0\leq u\leq\varepsilon$.
Thus, for $\varepsilon$ small enough, it follows from \eqref{3.48} that
\be\label{3.64}
\|J^{(2)}-J^{(k_0)}\|_{L^\infty((\mathcal{X}_{k_0}^\varepsilon)_4)}\leq C|\e^{\frac{\ii\gamma}{2}z^2}|t^{-1/2}\ln t.
\ee
Then we have
\be\label{3.65}
\|J^{(2)}-J^{(k_0)}\|_{L^1((\mathcal{X}_{k_0}^\varepsilon)_4)}\leq Ct^{-1}\ln t,\quad t>3,~\zeta\in\mathcal{I}.
\ee
By the general inequality $\|f\|_{L^p}\leq\|f\|^{1-1/p}_{L^\infty}\|f\|_{L^1}^{1/p}$, we find
\be\label{3.66}
\|J^{(2)}-J^{(k_0)}\|_{L^p((\mathcal{X}_{k_0}^\varepsilon)_4)}\leq Ct^{-1/2-1/2p}\ln t,\quad t>3,~\zeta\in\mathcal{I}.
\ee
The norms on $(\mathcal{X}_{k_0}^\varepsilon)_j$, $j=1,2,3$, are estimated in a similar way. Therefore, \eqref{3.59} follows.

If $k\in\partial D_\varepsilon(k_0)$, the variable $z=\sqrt{48tk_0}(k-k_0)$ tends to infinity as $t\rightarrow\infty$. It follows from \eqref{3.54} that
\berr
M^X(x,t;z)=I+\frac{M^X_1}{\sqrt{48tk_0}(k-k_0)}+O\bigg(\frac{1}{t}\bigg),\quad t\rightarrow\infty,~k\in \partial D_\varepsilon(k_0).
\eerr
Since
 $$M^{(k_0)}(x,t;k)=\eta^{-\hat{\sigma}}M^X(x,t;z),$$
thus we have
\be\label{3.67}
(M^{(k_0)})^{-1}(x,t;k)-I=-\frac{\eta^{-\hat{\sigma}}M^X_1}
{\sqrt{48tk_0}(k-k_0)}+O\bigg(\frac{1}{t}\bigg),\quad t\rightarrow\infty,~k\in \partial D_\varepsilon(k_0).
\ee
The estimate \eqref{3.60} immediately follows from \eqref{3.67} and $|M_1^X|\leq C$. By Cauchy's formula and \eqref{3.67}, we derive \eqref{3.61}.
\end{proof}

Let $$\mathcal{X}_{-k_0}^\varepsilon=\cup_{j=1}^4\Sigma_j\cap D_\varepsilon(-k_0),$$ where $\{\Sigma_j\}_1^{6}$ denote the restriction of $\Sigma$ to the contour labeled by $j$ in Fig. \ref{fig4}. Proceeding the analogous steps as above, we can approximate $M^{(2)}$ in the neighborhood $D_\varepsilon(-k_0)$ of $-k_0$ by $M^{(-k_0)}(x,t;k)$, where $M^{(-k_0)}$ satisfies the following RH problem:
\be\label{3.68}
\begin{aligned}
M^{(-k_0)}_+(x,t;k)&=M^{(-k_0)}_-(x,t;k)J^{(-k_0)}(x,t,z),~k\in \mathcal{X}_{-k_0}^\varepsilon,\\
M^{(-k_0)}(x,t;k)&\rightarrow I,\qquad\qquad\qquad\qquad\qquad\quad~ \text{as}~k\rightarrow\infty,
\end{aligned}
\ee
where the jump matrix $J^{(-k_0)}(x,t,z)$ is given by
\bea\label{3.69}
J^{(-k_0)}(x,t,z)=\hat{\eta}^{-\hat{\sigma}}J^Y(x,t,z), ~\hat{\eta}=(192tk_0^3)^{-\frac{\ii\nu}{2}}\e^{-8\ii tk_0^3+\chi(-k_0)},
\eea
and
\bea\label{3.70}
J^Y(x,t,z)&=&\left\{
\begin{aligned}
&\begin{pmatrix}
I ~& -\e^{-\frac{\ii z^2}{2}}(-z)^{-2\ii\nu}\frac{\rho^\dag(-k_0)}{1+\rho(-k_0)\rho^\dag(-k_0)}\\[4pt]
0 ~& 1
\end{pmatrix},~ z\in X_1,\\
&\begin{pmatrix}
I ~& 0\\[4pt]
\e^{\frac{\ii z^2}{2}}(-z)^{2\ii\nu}\rho(-k_0) ~& 1
\end{pmatrix},~~~~\qquad\qquad z\in X_2,\\
&\begin{pmatrix}
I ~& \e^{-\frac{\ii z^2}{2}}(-z)^{-2\ii\nu}\rho^\dag(-k_0)\\[4pt]
0 ~& 1
\end{pmatrix},\qquad\qquad z\in X_3,\\
&\begin{pmatrix}
I ~& 0\\[4pt]
-\e^{\frac{\ii z^2}{2}}(-z)^{2\ii\nu}\frac{\rho(-k_0)}{1+\rho(-k_0)\rho^\dag(-k_0)} ~& 1
\end{pmatrix},\qquad z\in X_4.
\end{aligned}
\right.
\eea
Using the formulae \eqref{2.28} and \eqref{3.8}, one verifies that
\be\label{3.71}
\begin{pmatrix}
0 ~& 1 ~& 0\\
1 ~& 0 ~& 0\\
0 ~& 0 ~& 1\\
\end{pmatrix}\overline{J^{Y}(x,t,-\bar{z})}\begin{pmatrix}
0 ~& 1 ~& 0\\
1 ~& 0 ~& 0\\
0 ~& 0 ~& 1\\
\end{pmatrix}=J^X(x,t,z),
\ee
which in turn implies, by uniqueness, that
\be\label{3.72}
M^{(-k_0)}(x,t;k)=\hat{\eta}^{-\hat{\sigma}}\begin{pmatrix}
0 ~& 1 ~& 0\\
1 ~& 0 ~& 0\\
0 ~& 0 ~& 1\\
\end{pmatrix}\overline{M^X(x,t,-\bar{z})}\begin{pmatrix}
0 ~& 1 ~& 0\\
1 ~& 0 ~& 0\\
0 ~& 0 ~& 1\\
\end{pmatrix}.
\ee
Furthermore, we have the similar results about $M^{(-k_0)}(x,t;k)$ as stated in Lemma \ref{lem3.3}. The jump matrix $J^{(-k_0)}$ satisfies the following estimates for $1\leq p\leq\infty$:
\be\label{3.59'}
\|J^{(2)}-J^{(-k_0)}\|_{L^p(\mathcal{X}_{-k_0}^\varepsilon)}\leq Ct^{-\frac{1}{2}-\frac{1}{2p}}\ln t,
\quad t>3,~\zeta\in\mathcal{I}.
\ee
As $t\rightarrow\infty$,
\be\label{3.60'}
\|(M^{(-k_0)})^{-1}(x,t;k)-I\|_{L^\infty(\partial D_\varepsilon(-k_0))}=O(t^{-1/2}),
\ee
and
\be\label{3.61'}
-\frac{1}{2\pi\ii}\int_{\partial D_\varepsilon(-k_0)}((M^{(-k_0)})^{-1}(x,t;k)-I)\dd k=\frac{\hat{\eta}^{-\hat{\sigma}}M^Y_1}
{\sqrt{48tk_0}}+O(t^{-1}),
\ee
where $M^Y_1$ is defined by
\be\label{3.62'}
M^Y_1=-\begin{pmatrix}
0 ~& 1 ~& 0\\
1 ~& 0 ~& 0\\
0 ~& 0 ~& 1\\
\end{pmatrix}\overline{M^X_1}\begin{pmatrix}
0 ~& 1 ~& 0\\
1 ~& 0 ~& 0\\
0 ~& 0 ~& 1\\
\end{pmatrix},~(M^Y_1)_{21}=-\ii\beta^Y,~\beta^Y=\overline{\beta^X}\begin{pmatrix}
0 ~& 1\\
1~& 0
\end{pmatrix}.
\ee

\subsection{Find asymptotic formula}
Define the approximate solution $M^{(app)}(x,t;k)$ by
\be\label{3.73}
M^{(app)}=\left\{\begin{aligned}
&M^{(k_0)},~\quad k\in D_\varepsilon(k_0),\\
&M^{(-k_0)},~~ k\in D_\varepsilon(-k_0),\\
&I,\qquad ~~~~{\text elsewhere}.
\end{aligned}
\right.
\ee
Let $\hat{M}(x,t;k)$ be
\be\label{3.74}
\hat{M}=M^{(2)}(M^{(app)})^{-1},
\ee
 then $\hat{M}(x,t;k)$ satisfies the following RH problem
\be\label{3.75}
\begin{aligned}
\hat{M}_+(x,t;k)&=\hat{M}_-(x,t;k)\hat{J}(x,t,k),\quad k\in\hat{\Sigma},\\
\hat{M}(x,t;k)&\rightarrow I,\qquad\qquad\qquad\qquad\quad\text{as}~k\rightarrow\infty,
\end{aligned}
\ee
where the jump contour $\hat{\Sigma}=\Sigma\cup\partial D_\varepsilon(k_0)\cup\partial D_\varepsilon(-k_0)$ is depicted in Fig. \ref{fig6}, and the jump matrix $\hat{J}(x,t,k)$ is given by
\be\label{3.76}
\hat{J}=\left\{
\begin{aligned}
&M^{(app)}_-J^{(2)}(M^{(app)}_+)^{-1},~~ k\in\hat{\Sigma}\cap (D_\varepsilon(k_0)\cup D_\varepsilon(-k_0)),\\
&(M^{(app)})^{-1},\qquad\qquad\quad k\in\partial D_\varepsilon(k_0)\cup\partial D_\varepsilon(-k_0),\\
&J^{(2)},\qquad\qquad\qquad\quad~~ k\in\hat{\Sigma}\setminus (\overline{D_\varepsilon(k_0)}\cup\overline{D_\varepsilon(-k_0)}).
\end{aligned}
\right.
\ee
\begin{figure}[htbp]
  \centering
  \includegraphics[width=3.5in]{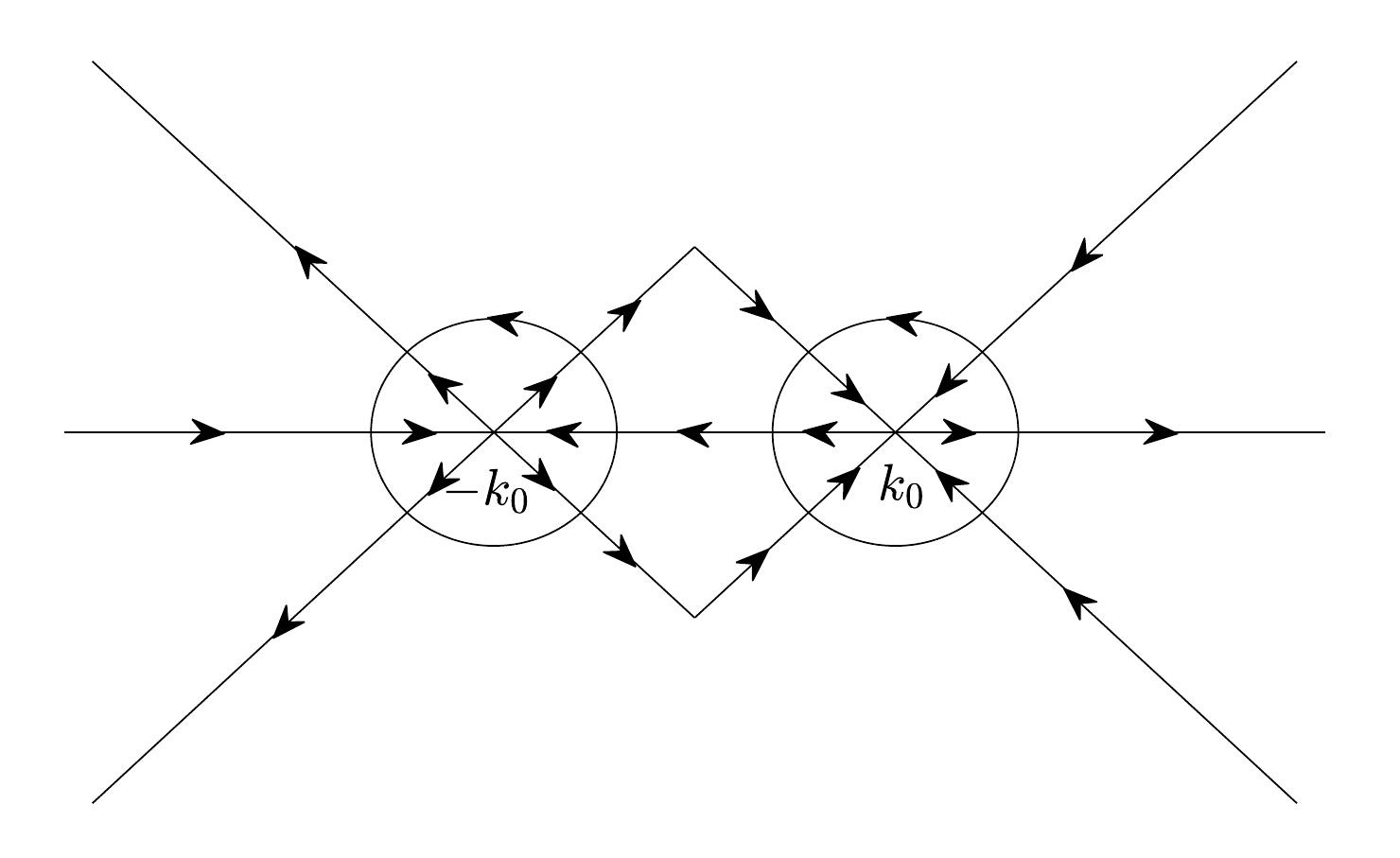}
  \caption{The contour $\hat{\Sigma}$.}\label{fig6}
\end{figure}
For convenience, we rewrite $\hat{\Sigma}$ as follows:
$$\hat{\Sigma}=(\partial D_\varepsilon(-k_0)\cup\partial D_\varepsilon(k_0))\cup(\mathcal{X}_{-k_0}^\varepsilon\cup\mathcal{X}_{k_0}^\varepsilon)\cup\hat{\Sigma}_1\cup\hat{\Sigma}_2,$$
where
$$\hat{\Sigma}_1=\bigcup_1^4\Sigma_j\setminus(D_\varepsilon(-k_0)\cup D_\varepsilon(k_0)),~\hat{\Sigma}_2=\bigcup_5^{6}\Sigma_j.$$
 Then we have the following lemma if we let $\hat{w}=\hat{J}-I$.
\begin{lemma}\label{lem3.4}
For $1\leq p\leq\infty$, the following estimates hold for $t>3$ and $\zeta\in\mathcal{I}$,
\bea
\|\hat{w}\|_{L^p(\partial D_\varepsilon(-k_0)\cup\partial D_\varepsilon(k_0))}&\leq& Ct^{-1/2},\label{3.77}\\
\|\hat{w}\|_{L^p(\mathcal{X}_{-k_0}^\varepsilon\cup\mathcal{X}_{k_0}^\varepsilon)}&\leq&Ct^{-\frac{1}{2}-\frac{1}{2p}}\ln t,\label{3.78}\\
\|\hat{w}\|_{L^p(\hat{\Sigma}_1)}&\leq&C\e^{-ct},\label{3.79}\\
\|\hat{w}\|_{L^p(\hat{\Sigma}_2)}&\leq&Ct^{-3/2}.\label{3.80}
\eea
\end{lemma}
\begin{proof}
The inequality \eqref{3.77} is a consequence of \eqref{3.60} and \eqref{3.76}.

For $k\in\mathcal{X}_{k_0}^\varepsilon$, we find $$\hat{w}=M^{(k_0)}_-(J^{(2)}-J^{(k_0)})(M^{(k_0)}_+)^{-1}.$$ Therefore, it follows from \eqref{3.59} that the estimate \eqref{3.78} holds.

For $k\in \Om_4\cap\Sigma_4$, let $k=k_0+u\e^{\frac{\ii\pi}{4}}$, $\varepsilon<u<\infty$, then
\berr
\text{Re}\Phi(\zeta,k)=4u^2(\sqrt{2}u+6k_0).
\eerr
Since $\hat{w}$ only has a nonzero $-\frac{\delta^{-1}\rho_{4,a}\e^{-t\Phi}}{\det\delta}$ in $(12)$ entry, hence, for $t\geq1$, by \eqref{3.11} and \eqref{3.18}, we get
\bea
|\hat{w}_{12}|&=&\bigg|-\frac{\delta^{-1}\rho_{4,a}\e^{-t\Phi}}{\det\delta}\bigg|
\leq C|\rho_{4,a}|\e^{-t|\text{Re}\Phi|}\nn\\
&\leq&\frac{C}{1+|k|^2}\e^{-\frac{3t}{4}|\text{Re}\Phi|}\leq  C\e^{-c\varepsilon^2t}.\nn
\eea
In a similar way, the other estimates on $\hat{\Sigma}_1$ hold. This proves \eqref{3.79}.

Since the matrix $\hat{w}$ on $\hat{\Sigma}_2$ only involves the small remainders $\{\rho_{j,r}\}_1^4$, thus, by Lemma \ref{lem3.1}, the estimate \eqref{3.80} follows.
\end{proof}
The estimates in Lemma \ref{lem3.4} imply that
\be\label{3.81}
\begin{aligned}
\|\hat{w}\|_{(L^1\cap L^2)(\hat{\Sigma})}&\leq Ct^{-1/2},\\
\|\hat{w}\|_{L^\infty(\hat{\Sigma})}&\leq Ct^{-1/2}\ln t,
\end{aligned}
\quad t>3,~ \zeta\in\mathcal{I}.
\ee
Let $\hat{C}$ denote the Cauchy operator associated with $\hat{\Sigma}$:
\berr
(\hat{C}f)(k)=\int_{\hat{\Sigma}}\frac{f(s)}{s-k}\frac{\dd s}{2\pi\ii},\quad k\in\bfC\setminus\hat{\Sigma},~f\in L^2(\hat{\Sigma}).
\eerr
We denote the boundary values of $\hat{C}f$ from the left and right sides of $\hat{\Sigma}$ by $\hat{C}_+f$ and $\hat{C}_-f$, respectively. As is well known, the operators $\hat{C}_\pm$ are bounded from $L^2(\hat{\Sigma})$ to $L^2(\hat{\Sigma})$, and $\hat{C}_+-\hat{C}_-=I$, here $I$ denotes the identity operator.

Define the operator $\hat{C}_{\hat{w}}$: $L^2(\hat{\Sigma})+L^\infty(\hat{\Sigma})\rightarrow L^2(\hat{\Sigma})$ by $\hat{C}_{\hat{w}}f=\hat{C}_-(f\hat{w}),$ that is, $\hat{C}_{\hat{w}}$ is defined by $\hat{C}_{\hat{w}}(f)=\hat{C}_+(f\hat{w}_-)+\hat{C}_-(f\hat{w}_+)$ where we have chosen, for simplicity, $\hat{w}_+=\hat{w}$ and $\hat{w}_-=0$. Then, by \eqref{3.81}, we find
\be\label{3.82}
\|\hat{C}_{\hat{w}}\|_{B(L^2(\hat{\Sigma}))}\leq C\|\hat{w}\|_{L^\infty(\hat{\Sigma})}\leq Ct^{-1/2}\ln t,
\ee
where $B(L^2(\hat{\Sigma}))$ denotes the Banach space of bounded linear operators $L^2(\hat{\Sigma})\rightarrow L^2(\hat{\Sigma})$. Therefore, there exists a $T>0$ such that $I-\hat{C}_{\hat{w}}\in B(L^2(\hat{\Sigma}))$ is invertible for all $\zeta\in\mathcal{I},$ $t>T$. Following this, we may define the $2\times2$ matrix-valued function $\hat{\mu}(x,t;k)$ whenever $t>T$ by
\be\label{3.83}
\hat{\mu}=I+\hat{C}_{\hat{w}}\hat{\mu}.
\ee
Then
\be\label{3.84}
\hat{M}(x,t;k)=I+\frac{1}{2\pi\ii}\int_{\hat{\Sigma}}\frac{(\hat{\mu}\hat{w})(x,t;s)}{s-k}\dd s,\quad k\in\bfC\setminus\hat{\Sigma}
\ee
is the unique solution of the RH problem \eqref{3.75} for $t>T$.

Moreover, standard estimates using the Neumann series shows that the function $\hat{\mu}(x,t;k)$ satisfies
\be\label{3.85}
\|\hat{\mu}(x,t;\cdot)-I\|_{L^2(\hat{\Sigma})}=O(t^{-1/2}),\quad t\rightarrow\infty,~\zeta\in\mathcal{I}.
\ee
In fact, equation \eqref{3.83} is equivalent to $\hat{\mu}=I+(I-\hat{C}_{\hat{w}})^{-1}\hat{C}_{\hat{w}}I$. Using the Neumann series, we get
$$\|(I-\hat{C}_{\hat{w}})^{-1}\|_{B(L^2(\hat{\Sigma}))}\leq\frac{1}{1-\|\hat{C}_{\hat{w}}\|_{B(L^2(\hat{\Sigma}))}}$$
whenever $\|\hat{C}_{\hat{w}}\|_{B(L^2(\hat{\Sigma}))}<1$. Thus, we find
\bea
\|\hat{\mu}(x,t;\cdot)-I\|_{L^2(\hat{\Sigma})}&=&\|(I-\hat{C}_{\hat{w}})^{-1}\hat{C}_{\hat{w}}I\|_{L^2(\hat{\Sigma})}\nn\\
&\leq&\|(I-\hat{C}_{\hat{w}})^{-1}\|_{B(L^2(\hat{\Sigma}))}\|\hat{C}_-(\hat{w})\|_{L^2(\hat{\Sigma})}\nn\\
&\leq&\frac{C\|\hat{w}\|_{L^2(\hat{\Sigma})}}{1-\|\hat{C}_{\hat{w}}\|_{B(L^2(\hat{\Sigma}))}}\leq C\|\hat{w}\|_{L^2(\hat{\Sigma})}\nn
\eea
for all $t$ large enough and all $\zeta\in\mathcal{I}$. In view of \eqref{3.81}, this gives \eqref{3.85}.
It follows from \eqref{3.84} that
\be\label{3.86}
\lim_{k\rightarrow\infty}k(\hat{M}(x,t;k)-I)=-\frac{1}{2\pi\ii}\int_{\hat{\Sigma}}(\hat{\mu}\hat{w})(x,t;k)\dd k.
\ee
Using \eqref{3.80} and \eqref{3.85}, we have
\bea
\int_{\hat{\Sigma}_2}(\hat{\mu}\hat{w})(x,t;k)\dd k&=&\int_{\hat{\Sigma}_2}\hat{w}(x,t;k)\dd k+\int_{\hat{\Sigma}_2}(\hat{\mu}(x,t;k)-I)\hat{w}(x,t;k)\dd k\nn\\
&\leq&\|\hat{w}\|_{L^1(\hat{\Sigma}_2)}+\|\hat{\mu}-I\|_{L^2(\hat{\Sigma}_2)}\|\hat{w}\|_{L^2(\hat{\Sigma}_2)}\nn\\
&\leq&Ct^{-3/2},\quad t\rightarrow\infty.\nn
\eea
Similarly, by \eqref{3.78} and \eqref{3.85}, the contribution from $\mathcal{X}_{-k_0}^\varepsilon\cup\mathcal{X}_{k_0}^\varepsilon$ to the right-hand side of \eqref{3.86} is $$O(\|\hat{w}\|_{L^1(\mathcal{X}_{-k_0}^\varepsilon\cup\mathcal{X}_{k_0}^\varepsilon)}+
\|\hat{\mu}-I\|_{L^2(\mathcal{X}_{-k_0}^\varepsilon\cup\mathcal{X}_{k_0}^\varepsilon)}
\|\hat{w}\|_{L^2(\mathcal{X}_{-k_0}^\varepsilon\cup\mathcal{X}_{k_0}^\varepsilon)})=O(t^{-1}\ln t),\quad t\rightarrow\infty.$$
By \eqref{3.79} and \eqref{3.85}, the contribution from $\hat{\Sigma}_1$ to the right-hand side of \eqref{3.86} is
\berr
O(\|\hat{w}\|_{L^1(\hat{\Sigma}_1)}+\|\hat{\mu}-I\|_{L^2(\hat{\Sigma}_1)}
\|\hat{w}\|_{L^2(\hat{\Sigma}_1)})=O(\e^{-ct}),\quad t\rightarrow\infty.
\eerr
Finally, by \eqref{3.61}, \eqref{3.61'}, \eqref{3.76}, \eqref{3.77} and \eqref{3.85}, we can get
\bea
&&-\frac{1}{2\pi\ii}\int_{\partial D_\varepsilon(-k_0)\cup\partial D_\varepsilon(k_0)}(\hat{\mu}\hat{w})(x,t;k)\dd k\nn\\
&=&-\frac{1}{2\pi\ii}\int_{\partial D_\varepsilon(-k_0)\cup\partial D_\varepsilon(k_0)}\hat{w}(x,t;k)\dd k-\frac{1}{2\pi\ii}\int_{\partial D_\varepsilon(-k_0)\cup\partial D_\varepsilon(k_0)}(\hat{\mu}(x,t;k)-I)\hat{w}(x,t;k)\dd k\nn\\
&=&-\frac{1}{2\pi\ii}\int_{\partial D_\varepsilon(-k_0)}\bigg((M^{(-k_0)})^{-1}(x,t;k)-I\bigg)\dd k-\frac{1}{2\pi\ii}\int_{\partial D_\varepsilon(k_0)}\bigg((M^{(k_0)})^{-1}(x,t;k)-I\bigg)\dd k\nn\\
&&+O(\|\hat{\mu}-I\|_{L^2(\partial D_\varepsilon(-k_0)\cup\partial D_\varepsilon(k_0))}
\|\hat{w}\|_{L^2(\partial D_\varepsilon(-k_0)\cup\partial D_\varepsilon(k_0))})\nn\\
&=&\frac{\eta^{-\hat{\sigma}}M^X_1}
{\sqrt{48tk_0}}+\frac{\hat{\eta}^{-\hat{\sigma}}M^Y_1}
{\sqrt{48tk_0}}+O(t^{-1}),\quad t\rightarrow\infty.\nn
\eea
Thus, we obtain the following important relation
\be\label{3.91}
\lim_{k\rightarrow\infty}k(\hat{M}(x,t;k)-I)=\frac{\eta^{-\hat{\sigma}}M^X_1+\hat{\eta}^{-\hat{\sigma}}M^Y_1}
{\sqrt{48t k_0}}+O(t^{-1}\ln t),\quad t\rightarrow\infty.
\ee

Taking into account that \eqref{2.29}, \eqref{3.3}, \eqref{3.4}, \eqref{3.36} and \eqref{3.91}, for sufficient large $k\in\bfC\setminus\hat{\Sigma}$, we get
\bea\label{4.92}
\begin{pmatrix}
\bar{u}(x,t)~& u(x,t)
\end{pmatrix}&=&2\ii\lim_{k\rightarrow\infty}(kM(x,t;k))_{21}\nn\\
&=&2\ii\lim_{k\rightarrow\infty}k(\hat{M}(x,t;k)-I)_{21}\nn\\
&=&2\ii\bigg(\frac{-\ii\eta^2\beta^X-\ii\hat{\eta}^2\beta^Y}
{\sqrt{48tk_0}}\bigg)+O\bigg(\frac{\ln t}{t}\bigg).\nn
\eea

Collecting the above computations, we obtain our main results stated as the following theorem.
\begin{theorem}\label{th3.2}
Let $u_0(x)$ lie in the Schwartz space $S(\bfR)$. Then, for any positive constant $M>1$, as $t\rightarrow\infty$, the solution $u(x,t)$ of the initial problem for Sasa-Satsuma equation \eqref{1.1} on the line satisfies the following asymptotic formula
\be
u(x,t)=\frac{u_{as}(x,t)}{\sqrt{t}}+O\bigg(\frac{\ln t}{t}\bigg),\quad t\rightarrow\infty,~\zeta\in\mathcal{I},
\ee
where the error term is uniform with respect to $x$ in the given range, and the leading-order
coefficient $u_{as}(x,t)$ is given by
\bea
u_{as}(x,t)&=&\frac{\nu\e^{-\frac{\pi\nu}{2}}}{\sqrt{24\pi k_0}}\bigg((192tk_0^3)^{\ii\nu}\e^{16\ii tk_0^3+2\chi(k_0)+\frac{\pi\ii}{4}}\Gamma(-\ii\nu)\frac{s_{32}(k_0)}{s_{33}(k_0)}\nn\\
&&+(192tk_0^3)^{-\ii\nu}\e^{-16\ii tk_0^3+2\chi(-k_0)-\frac{\pi\ii}{4}}\Gamma(\ii\nu)\frac{\bar{s}_{31}(k_0)}{\bar{s}_{33}(k_0)}\bigg),
\eea
where $k_0$, $\nu$ and $\chi(k)$ are given by \eqref{3.1}, \eqref{3.6} and \eqref{3.7}, respectively.
\end{theorem}

{\bf Acknowledgments.}

This work was supported in part by the National Natural Science Foundation of
China under grants 11731014, 11571254 and 11471099.

\appendix
\section{Appendix. Proof of Theorem \ref{th3.1}}
\setcounter{equation}{0}
To solve the model RH problem \eqref{3.53}, it is convenient to introduce the following transformation
\berr
\Psi(z)=M^X(z)z^{\ii\nu\sigma}\e^{\frac{\ii z^2}{4}\sigma},
\eerr
which implies that
\be
\Psi_+(z)=\Psi_-(z)v(k_0),~v(k_0)=z^{-\ii\nu\hat{\sigma}}\e^{-\frac{\ii z^2}{4}\hat{\sigma}}J^X,
\ee
where we suppress the $x$ and $t$ dependence for clarity. Since the jump matrix is constant along each ray, we have
\berr
\frac{\dd\Psi_+(z)}{\dd z}=\frac{\dd\Psi_-(z)}{\dd z}v(k_0),
\eerr
from  which it follows that $\frac{\dd\Psi(z)}{\dd z}\Psi^{-1}(z)$ has no jump discontinuity along any of the rays. Moreover,
\bea
\frac{\dd\Psi(z)}{\dd z}\Psi^{-1}(z)&=&\frac{\dd M^X(z)}{\dd z}(M^X)^{-1}(z)+\frac{\ii}{2}zM^X(z)\sigma(M^X)^{-1}(z)+\frac{\ii\nu}{2}M^X(z)\sigma(M^X)^{-1}(z)\nn\\
&=&\frac{\ii z}{2}\sigma-\frac{\ii}{2}[\sigma,M_1^X]+O\bigg(\frac{1}{z}\bigg).\nn
\eea
It follows by Liouville's argument that
\be\label{A.2}
\frac{\dd\Psi(z)}{\dd z}=\bigg(\frac{\ii z}{2}\sigma+\beta\bigg)\Psi(z),
\ee
where
\berr
\beta=-\frac{\ii}{2}[\sigma,M_1^X]=\begin{pmatrix}
0~& \beta_{12}\\
\beta_{21}~& 0\\
\end{pmatrix}.
\eerr
In particular,
\be
(M_1^X)_{21}=-\ii\beta_{21}.
\ee

It is possible to show that the solution of the RH problem \eqref{3.53} for $M^X(z)$ is unique since $\det J^X=1$, and therefore we have $(M^X(\bar{k}))^\dag=(M^X(k))^{-1}$, which implies that $$\beta_{12}=-\beta^\dag_{21}.$$ From equation \eqref{A.2} we obtain
\be\label{A.4}
\begin{aligned}
\frac{\dd^2\Psi_{22}(z)}{\dd z^2}&=\bigg(-\frac{\ii}{2}-\frac{z^2}{4}+\beta_{21}\beta_{12}\bigg)\Psi_{22}(z),\\
\beta_{21}\Psi_{12}(z)&=\frac{\dd\Psi_{22}(z)}{\dd z}+\frac{\ii}{2}z\Psi_{22},\\
\frac{\dd^2\beta_{21}\Psi_{11}(z)}{\dd z^2}&=\bigg(\frac{\ii}{2}-\frac{z^2}{4}+\beta_{21}\beta_{12}\bigg)\beta_{21}\Psi_{11}(z),\\
\Psi_{21}(z)&=\frac{1}{\beta_{21}\beta_{12}}\bigg(\frac{\dd\beta_{21}\Psi_{11}(z)}{\dd z}-\frac{\ii}{2}z\beta_{21}\Psi_{11}\bigg).
\end{aligned}
\ee
It is well known that the Weber's equation
\berr
\frac{\dd^2g(s)}{\dd s^2}+\bigg(\frac{1}{2}-\frac{s^2}{4}+a\bigg)g(s)=0
\eerr
has the solution
\be
g(s)=c_1D_a(s)+c_2D_a(-s),
\ee
where $D_a(\cdot)$ denotes the standard parabolic-cylinder function and satisfies
\bea
\frac{\dd D_a(s)}{\dd s}&+&\frac{s}{2}D_a(s)-aD_{a-1}(s)=0,\label{A.6}\\
D_{a-1}(s)&=&\frac{\Gamma(a)}{\sqrt{2\pi}}\bigg(\e^{\frac{\ii\pi}{2}(a-1)}D_{-a}(\ii s)+\e^{-\frac{\ii\pi}{2}(a-1)}D_{-a}(-\ii s)\bigg).\label{A.7}
\eea
Furthermore, as $s\rightarrow\infty$, we know from \cite{WW} that
\berr
D_a(s)=\left\{\begin{aligned}
&s^a\e^{-\frac{s^2}{4}}(1+O(s^{-2})),\qquad\qquad\qquad\qquad\quad\qquad\qquad\qquad\qquad|\arg s|<\frac{3\pi}{4},\\
&s^a\e^{-\frac{s^2}{4}}(1+O(s^{-2}))-\frac{\sqrt{2\pi}}{\Gamma(-a)}\e^{a\pi\ii+\frac{s^2}{4}}s^{-a-1}(1+O(s^{-2})),\qquad~\frac{\pi}{4}<\arg s<\frac{5\pi}{4},\\
&s^a\e^{-\frac{s^2}{4}}(1+O(s^{-2}))-\frac{\sqrt{2\pi}}{\Gamma(-a)}\e^{-a\pi\ii+\frac{s^2}{4}}s^{-a-1}(1+O(s^{-2})),~-\frac{5\pi}{4}<\arg s<-\frac{\pi}{4},
\end{aligned}
\right.
\eerr
where $\Gamma$ is the Gamma function. Setting $a=\ii\beta_{21}\beta_{12}$, we have
\bea
\Psi_{22}(z)&=&c_1D_a(\e^{-\frac{3\ii\pi}{4}}z)+c_2D_a(\e^{\frac{\pi\ii}{4}}z),\nn\\
\beta_{21}\Psi_{11}(z)&=&c_3D_{-a}(\e^{\frac{3\ii\pi}{4}}z)+c_4D_{-a}(\e^{-\frac{\pi\ii}{4}}z).\nn
\eea
Since as $z\rightarrow\infty$, we have
\be
\Psi_{11}\rightarrow z^{\ii\nu}\e^{\frac{\ii z^2}{4}}I,\quad \Psi_{22}\rightarrow z^{-\ii\nu}\e^{-\frac{\ii z^2}{4}},
\ee
hence, as $\arg z\in(-\frac{3\pi}{4},-\frac{\pi}{4})$, we immediately arrive that
\be
\begin{aligned}
\Psi_{22}(z)&=\e^{-\frac{\pi\nu}{4}}D_a(\e^{\frac{\pi\ii}{4}}z),~a=-\ii\nu,\\
\beta_{21}\Psi_{11}(z)&=\beta_{21}\e^{\frac{3\pi\nu}{4}}D_{-a}(\e^{\frac{3\pi\ii}{4}}z).
\end{aligned}
\ee
According to \eqref{A.4} and \eqref{A.6}, we have
\be
\begin{aligned}
\Psi_{21}(z)&=\beta_{21}\e^{\frac{\pi(\ii+3\nu)}{4}}D_{-a-1}(\e^{\frac{3\pi\ii}{4}}z),\\
\beta_{21}\Psi_{12}(z)&=\e^{\frac{\pi(\ii-\nu)}{4}}aD_{a-1}(\e^{\frac{\pi\ii}{4}}z).
\end{aligned}
\ee

Similarly, for $\arg z\in(-\frac{\pi}{4},\frac{\pi}{4})$, we can get
\be
\begin{aligned}
\Psi_{22}(z)&=\e^{-\frac{\pi\nu}{4}}D_a(\e^{\frac{\pi\ii}{4}}z),~a=-\ii\nu,\\
\beta_{21}\Psi_{11}(z)&=\beta_{21}\e^{-\frac{\pi\nu}{4}}D_{-a}(\e^{-\frac{\pi\ii}{4}}z),\\
\Psi_{21}(z)&=\beta_{21}\e^{-\frac{\pi(3\ii+\nu)}{4}}D_{-a-1}(\e^{-\frac{\pi\ii}{4}}z),\\
\beta_{21}\Psi_{12}(z)&=\e^{\frac{\pi(\ii-\nu)}{4}}aD_{a-1}(\e^{\frac{\pi\ii}{4}}z).
\end{aligned}
\ee
Along the ray $\arg z=-\frac{\pi}{4}$, we have
\berr
\Psi_+(z)=\Psi_-(z)\begin{pmatrix}
I ~& 0\\
-\rho(k_0) ~& 1
\end{pmatrix},
\eerr
thus,
\be\label{A.12}
\beta_{12}\e^{\frac{3\pi\nu}{4}}D_{-a}(\e^{\frac{3\pi\ii}{4}}z)=\beta_{21}\e^{-\frac{\pi\nu}{4}}D_{-a}(\e^{-\frac{\pi\ii}{4}}z)
-\e^{\frac{\pi(\ii-\nu)}{4}}aD_{a-1}(\e^{\frac{\pi\ii}{4}}z)\rho(k_0).
\ee
On the other hand, it follows from \eqref{A.7} that
\be\label{A.13}
D_{a-1}(\e^{\frac{\pi\ii}{4}}z)=\frac{\Gamma(a)}{\sqrt{2\pi}}
\bigg(\e^{\frac{\ii\pi}{2}(a-1)}D_{-a}(\e^{\frac{3\pi\ii}{4}}z)+\e^{-\frac{\ii\pi}{2}(a-1)}D_{-a}(\e^{-\frac{\pi\ii}{4}}z)\bigg).
\ee
Compared the coefficients of \eqref{A.12} with \eqref{A.13}, we can find that
\be
\beta_{21}=\frac{\Gamma(-\ii\nu)}{\sqrt{2\pi}}\e^{\frac{\pi\ii}{4}-\frac{\pi\nu}{2}}\nu\rho(k_0).
\ee
The estimate \eqref{3.56} is an consequence of Lemma A.4 in \cite{JL3} and the asymptotic formula \eqref{3.54}.
\medskip
\small{

}
\end{document}